\documentclass[11pt,a4paper,lualatex]{amsart}

\usepackage[marginparwidth=0pt,margin=20truemm]{geometry} 

\usepackage{mypackage} 
\usepackage{mycommand} 
\usetikzlibrary{knots}
\usetikzlibrary{positioning}

\DeclareMathOperator{\WRT}{WRT}
\DeclareMathOperator{\diag}{diag}

\usepackage[luatex, pdfencoding=auto,hypertexnames=false]{hyperref}
\hypersetup{colorlinks=false}

\numberwithin{equation}{section} 
\usepackage{mytheoremeng} 

\begin{document}


\title[WRT invariants and homological blocks for plumbed homology spheres]{Witten--Reshetikhin--Turaev invariants and homological blocks for plumbed homology spheres}
\author[Y. Murakami]{Yuya Murakami}
\address{Faculty of Mathematics, Kyushu University, 744, Motooka, Nishi-ku, Fukuoka, 819-0395, Japan}
\email{murakami.yuya.896@m.kyushu-u.ac.jp}


\date{\today}

\maketitle


\begin{abstract}
	In this paper, we prove a conjecture by Gukov--Pei--Putrov--Vafa for a wide class of plumbed $ 3 $-manifolds.
	Their conjecture states that Witten--Reshetikhin--Turaev (WRT) invariants are radial limits of homological blocks, which are $ q $-series introduced by them for plumbed $ 3 $-manifolds with negative definite linking matrices.
	The most difficult point in our proof is to prove the vanishing of weighted Gauss sums that appear in coefficients of negative degree in asymptotic expansions of homological blocks.
	To deal with it, we develop a new technique for asymptotic expansions, which enables us to compare asymptotic expansions of rational functions and false theta functions related to WRT invariants and homological blocks, respectively.
	In our technique, our vanishing results follow from holomorphy of such rational functions.
\end{abstract}


\tableofcontents


\section{Introduction} \label{sec:intro}


In this paper, we study Witten--Reshetikhin--Turaev (WRT) invariants.
They are quantum invariants of $ 3 $-manifolds constructed by Witten~\cite{Witten} from a physical viewpoint and constructed by Reshetikhin--Turaev~\cite[Theorem 3.3.2]{Reshetikhin-Turaev} from a mathematical viewpoint.
The Witten's asymptotic expansion conjecture plays a central role in studying WRT invariants.
It claims that the Chern-Simons invariants and the Reidemeister torsions appear in the asymptotic expansion of WRT invariants.
This conjecture is proved by Lawrence--Zagier~\cite{LZ} for the Poincar\'{e} homology sphere, Hikami~\cite{H_Bries} for Brieskorn homology spheres, Hikami~\cite{H_Seifert} and Matsusaka--Terashima~\cite{Matsusaka-Terashima} independently for Seifert homology spheres.
Many other authors have studied this conjecture~\cite{Andersen-Himpel,Andersen-Mistegard,Andersen,Beasley-Witten,Charles,Chun,Charles-Marche_II,FIMT,Freed-Gompf,GMP,H_Lattice,H_Lattice2,Hansen-Takata,Jeffrey,Rozansky1,Rozansky2,Wu}.

Recently, Gukov--Pei--Putrov--Vafa~\cite{GPPV} introduced $ q $-series invariants called homological blocks or GPPV invariants for plumbed manifolds.
They conjectured that homological blocks have good modular transformations and their radial limits yield WRT invariants.
These conjectures show a roadmap to prove the Witten's asymptotic expansion conjecture.
The first conjecture, modular transformations of homological blocks, are related to quantum modular forms introduced by Zagier~\cite{Zagier_quantum}.
This conjecture is proved by Matsusaka--Terashima~\cite{Matsusaka-Terashima} for Seifert homology spheres and Bringmann--Mahlburg--Milas~\cite{BMM_high_depth} for non-Seifert homology spheres whose surgery diagrams are the H-graphs.
Bringmann--Nazaroglu~\cite{Bringmann-Nazaroglu} and Bringmann--Kaszian--Milas--Nazaroglu~\cite{BKMN_False_modular} clarified and proved the modular transformation formulas of false theta functions, which are main tools to study modular transformations of homological blocks.

The second conjecture, which asks the relation between WRT invariants and radial limits of homological blocks, is proved by Fuji--Iwaki--Murakami--Terashima~\cite{FIMT} with a result in Andersen--Misteg{\aa}rd~\cite{Andersen-Mistegard} for Seifert homology spheres and Mori--Murakami~\cite{MM} for non-Seifert homology spheres whose surgery diagrams are the H-graphs.
In this paper, we prove this conjecture for a large class of plumbed manifolds by developing some new techniques.

Let us explain our setting.
Let $ \Gamma = (V, E, (w_v)_{v \in V}) $ be a plumbing graph, that is, a finite tree with the vertex set $ V $, the edge set $ E $, and integral weights $ w_v \in \Z $ for each vertex $ v \in V $.
Following \cite{GPPV}, we assume that the linking matrix $ W $ of $ \Gamma $ is negative definite.
Then one can define the homological block $ \widehat{Z}_\Gamma (q) $ with $ q \in \bbC $ and $ \abs{q} < 1 $. 
Let $ M(\Gamma) $ be the plumed 3-manifold obtained from $ S^3 $ through the surgery along the diagram defined by $ \Gamma $.
We also assume that $ M(\Gamma) $ is an integral homology sphere.
This condition is equivalent to $ \det W = \pm 1 $ since $ H_1(M(\Gamma), \Z) \cong \Z^V / W(\Z^V)$ by use of the Mayer--Vietoris sequence. 
For a positive integer $ k $, let $ \WRT_k(M(\Gamma)) $ be the WRT invariant of $ M(\Gamma) $ normalised as $ \WRT_k(S^3) = 1 $. 
We also denote $ \zeta_k \coloneqq e^{2\pi\iu/k} $.

In the above setting, we can state the conjecture of Gukov--Pei--Putrov--Vafa~\cite{GPPV} by the following.

\begin{conj}[{\cite[Conjecture 2.1, Equation (A.28)]{GPPV}}] \label{conj:GPPV}
	\[
	\WRT_k(M(\Gamma))
	= \frac{1}{2(\zeta_{2k} - \zeta_{2k}^{-1})}
	\lim_{q \to \zeta_k} \widehat{Z}_\Gamma \left( q \right).
	\]
\end{conj}

In this paper, we prove this conjecture for a wide class of plumbing graphs.

\begin{thm} \label{thm:main}
	\cref{conj:GPPV} is true for plumbing graphs depicted in \cref{fig:main}.
\end{thm}

\begin{figure}	
	\centering
	\begin{tikzpicture}
		\node[shape=circle,fill=black, scale = 0.4] (1) at (0,0) { };
		\node[shape=circle,fill=black, scale = 0.4] (2) at (1.5,0) { };
		\node[shape=circle,fill=black, scale = 0.4] (3) at (-1,1) { };
		\node[shape=circle,fill=black, scale = 0.4] (4) at (-1,0.5) { };
		\node[shape=circle,fill=black, scale = 0.4] (5) at (-1,-1) { };
		\node[shape=circle,fill=black, scale = 0.4] (6) at (2.5,1) { };
		\node[shape=circle,fill=black, scale = 0.4] (7) at (2.5,0.5) { };
		\node[shape=circle,fill=black, scale = 0.4] (8) at (2.5,-1) { };
		
		\node[draw=none] (B1) at (0,0.4) {$ w_1 $};
		\node[draw=none] (B2) at (1.5, 0.4) {$ w_2 $};
		\node[draw=none] (B3) at (-1.4,1) {$ w_3 $};
		\node[draw=none] (B4) at (-1.4,0.5) {$ w_4 $};
		\node[draw=none] (B5) at (-1.4,-1) {$ w_{N'} $};		
		\node[draw=none] (B6) at (3.2,1) {$ w_{N'+1} $};	
		\node[draw=none] (B7) at (3.2,0.5) {$ w_{N'+2} $};		
		\node[draw=none] (B8) at (2.9,-1) {$ w_N $};	
		\node[draw=none,rotate=270] (B9) at (-1.4,-0.2) {$\cdots$};		
		\node[draw=none,rotate=270] (B10) at (2.9,-0.2) {$\ldotp\ldotp\ldotp\ldotp$};
		
		\path [-](1) edge node[left] {} (2);
		\path [-](1) edge node[left] {} (3);
		\path [-](1) edge node[left] {} (4);
		\path [-](1) edge node[left] {} (5);
		\path [-](2) edge node[left] {} (6);
		\path [-](2) edge node[left] {} (7);
		\path [-](2) edge node[left] {} (8);		
	\end{tikzpicture}
	\caption{} \label{fig:main}
\end{figure}

Here we remark that a plumbing graph has the form depicted in \cref{fig:main} if and only if $ \abs{\overline{v}} + 2 - \deg(v) > 0 $ for any vertex $ v \in V $, where $ \deg(v) $ is the degree of $ v $ and
$ \overline{v} \coloneqq \{ i \in V \mid \{i, v \} \in E, \, \deg(i) = 1 \} $.

In a proof of \cref{thm:main}, the most difficult point is to prove the vanishing of weighted Gauss sums that appear in coefficients of negative degree in asymptotic expansions of homological blocks.
The previous works~\cite{BMM_high_depth,MM} deal with this difficulty by direct calculations (\cite[Theorem 4.1]{BMM_high_depth},\cite[Proposition 4.2]{MM}).
However, we prove it indirectly by using our asymptotic formula (\cref{prop:Euler--Maclaurin_poly} and \cref{cor:asymp_F(f; t)}) and holomorphy of a rational function whose radial limits are WRT invariants.
In a sense, our technique is a generalisation of the method of Lawrence--Zagier~\cite[pp.98, Proposition]{LZ} using $ L $-funcitons for periodic maps.

This paper will be organised as follows. 
In \cref{sec:fund_data}, we prepare some notations for plumbing graphs  $ \Gamma $ which we use throughout this paper.
In \cref{sec:WRT}, we calculate WRT invariants for plumbed homology spheres $ M(\Gamma) $.
The point of our calculation is to represent WRT invariants as a sum for the submatrix of $ W^{-1} $ with vertices whose degrees are greater than $ 2 $.
In \cref{sec:HB_false}, we express homological blocks as false theta functions.
Then we can study the asymptotic expansions of homological blocks as $ q \to \zeta_k $.
To calculate asymptotic expansions, we develop a formula in \cref{subsec:Euler--Maclaurin} by the Euler--Maclaurin summation formula based on Zagier~\cite[Equation (44)]{Zagier_asymptotic}.
In \cref{subsec:asymp_Gauss}, we develop the important asymptotic formula, which plays a central role in proving our main theorem.
This formula asserts that the same factors appear in the asymptotic expansions of homological blocks and rational functions whose radial limits are WRT invariants.
Finally, we prove our main theorem in \cref{sec:proof_main}.


\section*{Acknowledgement} \label{sec:acknowledgement}


The author would like to show the greatest appreciation to Takuya Yamauchi for giving many pieces of advice. 
The author would like to thank Takuya Yamauchi, Yuji Terashima, Kazuhiro Hikami, Toshiki Matsusaka for giving many comments. 
The author is deeply grateful to Akihito Mori for a lot of discussion. 
I thank the referees for their helpful suggestions and comments which improved the presentation of our paper.
The author is supported by JSPS KAKENHI Grant Number JP 20J20308 and JP23KJ1675.


\section{Basic notations for plumbing graphs} \label{sec:fund_data}


In this section, we list some notations for plumbing graphs and their basic properties, which we use throughout this paper.


\subsection{Notations for graphs} \label{subsec:graph}


In this subsection, we prepare settings for graphs.
As in \cref{sec:intro}, let $ \Gamma = (V, E, (w_v)_{v \in V}) $ be a plumbing graph and $ W $ be its linking matrix such that it is negative definite and $ \det W = \pm 1 $.
Here we consider the edge set $ E $ as the subset of $ \{ \{ v, v' \} \mid v, v' \in V \} $.
We identify $ M_{\abs{V}}(\Z) $ and $ \End(\Z^V) $ and consider $ W $ as the element of $ \End(\Z^V) $.
We remark that $ w_v \in \Z_{<0} $ for any vertex $ v \in V $.

For two plumbing graphs $ \Gamma $ and $ \Gamma' $, Neumann (\cite[Proposition 2.2]{Neumann_Lecture}, \cite[Theorem 3.1]{Neumann_work}) proved that two $ 3 $-manifolds $ M(\Gamma) $ and $ M(\Gamma') $ are homeomorphic if and only if $ \Gamma $ and $ \Gamma' $ are related by Neumann moves shown in \cref{fig:Neumann}.
Thus, we can assume $ w_i \le -2 $ for a vertex $ i $ with $ \deg(i) = 1 $. 

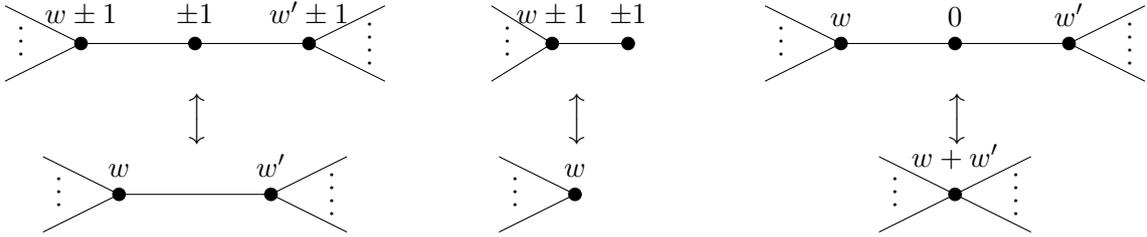
\begin{figure}[htp]
	\centering
	\begin{tikzpicture}
		\draw[fill]
		(-1.5,0) node[above=0.1cm]{$w \pm 1$} circle(0.5ex)--
		(0,0) node[above=0.1cm]{$\pm 1$} circle(0.5ex)--
		(1.5,0) node[above=0.1cm]{$w' \pm 1$} circle(0.5ex)
		(-2.5,0.5) node[above]{}--(-1.5,0) node[above]{}
		(-2.3,0) node[rotate=270]{$\cdots$}
		(-2.5,-0.5) node[above]{}--(-1.5,-0) node[above]{}
		(1.5,0) node[above]{}--(2.5,0.5) node[above]{}
		(2.3,0) node[rotate=270]{$\ldotp\ldotp\ldotp\ldotp$}
		(1.5,0) node[above]{}--(2.5,-0.5) node[above]{}
		(0,-1) node[rotate=270]{$\longleftrightarrow$}
		(-1,-2) node[above=0.1cm]{$ w $} circle(0.5ex)--
		(1,-2) node[above=0.1cm]{$ w' $} circle(0.5ex)
		(-2,-1.5) node[above]{}--(-1,-2) node[above]{}
		(-1.8,-2) node[rotate=270]{$\cdots$}
		(-2,-2.5) node[above]{}--(-1,-2) node[above]{}
		(1,-2) node[above]{}--(2,-1.5) node[above]{}
		(1.8,-2) node[rotate=270]{$\ldotp\ldotp\ldotp\ldotp$}
		(1,-2) node[above]{}--(2,-2.5) node[above]{};
		\draw[fill]
		(4.7,0) node[above=0.1cm]{$w \pm 1$} circle(0.5ex)--
		(5.7,0) node[above=0.1cm]{$\pm 1$} circle(0.5ex)
		(3.9,0.5) node[above]{}--(4.7,0) node[above]{}
		(4.1,0) node[rotate=270]{$\cdots$}
		(3.9,-0.5) node[above]{}--(4.7,0) node[above]{}
		(5,-1) node[rotate=270]{$\longleftrightarrow$}
		(5,-2) node[above=0.1cm]{$w$} circle(0.5ex)
		(4,-1.5) node[above]{}--(5,-2) node[above]{}
		(4.2,-2) node[rotate=270]{$\cdots$}
		(4,-2.5) node[above]{}--(5,-2) node[above]{};
		\draw[fill]
		(8.5,0) node[above=0.1cm]{$w$} circle(0.5ex)--
		(10,0) node[above=0.1cm]{$0$} circle(0.5ex)--
		(11.5,0) node[above=0.1cm]{$w'$} circle(0.5ex)
		(7.5,0.5) node[above]{}--(8.5,0) node[above]{}
		(7.7,0) node[rotate=270]{$\cdots$}
		(7.5,-0.5) node[above]{}--(8.5,0) node[above]{}
		(11.5,0) node[above]{}--(12.5,0.5) node[above]{}
		(12.3,0) node[rotate=270]{$\ldotp\ldotp\ldotp\ldotp$}
		(11.5,0) node[above]{}--(12.5,-0.5) node[above]{}
		(10,-1) node[rotate=270]{$\longleftrightarrow$}
		(10,-2) node[above=0.2cm]{$w + w'$} circle(0.5ex)
		(9,-1.5) node[above]{}--(10,-2) node[above]{}
		(9.2,-2) node[rotate=270]{$\cdots$}
		(9,-2.5) node[above]{}--(10,-2) node[above]{}
		(10,-2) node[above]{}--(11,-1.5) node[above]{}
		(10.8,-2) node[rotate=270]{$\ldotp\ldotp\ldotp\ldotp$}
		(10,-2) node[above]{}--(11,-2.5) node[above]{};
	\end{tikzpicture}
	\caption{Neumann moves}
	\label{fig:Neumann}
\end{figure}

For a positive integer $ d $, let
\[
V_d \coloneqq \{ v \in V \mid \deg(v) = d \}, \quad
V_{\ge d} \coloneqq \{ v \in V \mid \deg(v) \ge d \}.
\]
Let $ W_d \in \End(\Z^{V_d}) $ and $ W_{\ge d} \in \End(\Z^{V_{\ge d}}) $ be the submatrices of $ W $ with $ V_d \times V_d $ and $ V_{\ge d} \times V_{\ge d} $ components respectively.
Moreover, for a positive integer $ e $, let $ W_{d, \ge e} \in \Hom(\Z^{V_{\ge e}}, \Z^{V_d}) $ and $ W_{\ge d, e} \in  \Hom(\Z^{V_e}, \Z^{V_{\ge d}}) $ be the submatrices of $ W $ with $ V_d \times V_{\ge e} $ and $ V_{\ge d} \times V_e $ components respectively.

We need to focus on vertices with degree $ 1 $ to calculate WRT invariants and homological blocks.
For this reason, we define 
\[
\overline{v} \coloneqq \{ i \in V_1 \mid \{ i, v \} \in E \}, \quad
M_v \coloneqq \prod_{i \in \overline{v}} w_i
\]
for each vertex $ v \in V_{\ge 2} $.
Here we define $ M_v \coloneqq 1 $ if $ \overline{v} = \emptyset $.

The condition $ \det W = \pm 1 $ implies the following lemma, which Akihito Mori told the author.

\begin{lem} \label{lem:coprime}
	For a vertex $ v \in V_{\ge 2} = \{ v \in V \mid \deg v \ge 2 \} $ and distinct vertices $ i, j \in \overline{v} $, it holds that $ \gcd(w_i, w_j) = 1 $.
\end{lem}

\begin{proof}
	Let $ \{ e_{v} \}_{v \in V} $ be the standard basis of $ \R^V $.
	Since 
	\[
	W e_i - W e_j 
	=
	(e_{v} + w_i e_i) - (e_{v} + w_j e_j)
	=
	w_i e_i - w_j e_j,
	\]
	$ \gcd(w_i, w_j) $ divides $ \det W = \pm 1 $.
\end{proof}


\subsection{The inverse matrix of the linking matrix} \label{subsec:S}


In this subsection, we calculate the inverse matrix of the linking matrix.

Let $ S \in \Aut(\Q^{V_{\ge 2}}) \cap \End(\Z^{V_{\ge 2}}) $ be the $ V_{\ge 2} \times V_{\ge 2} $ submatrix of $ -W^{-1} \in \Aut(\Q^V) $.
Here we recall $ V_{\ge 2} = \{ v \in V \mid \deg v \ge 2 \} $.
We also denote $ T \coloneqq  -W_1^{-1} W_{1, \ge 2} $. 
Here we remark that $ S $ is positive definite since $ W $ is negative definite.

The inverse matrices $ W^{-1} $ and $ S^{-1} $ have the following properties.

\begin{prop} \label{prop:S}
	\begin{enumerate}
		\item \label{item:prop:S1}		
		\[
		-W^{-1} = \pmat{T \\ I} S \pmat{{}^t\!T & I} - \pmat{W_1^{-1} & \\ & O}.
		\]	
		\item \label{item:prop:S2}
		\[
		\det S =  \det W \prod_{i \in V_1} w_i
		= \prod_{i \in V_1} \abs{w_i}.
		\]
		\item \label{item:prop:S3}
		\[
		S^{-1} = -W_{\ge 2} + \diag \left( \sum_{i \in \overline{v}} \frac{1}{w_i} \right)_{v \in V_{\ge 2}}.
		\]
		Here, we denote $ \diag(a_v)_{v \in V_{\ge 2}} $ the diagonal matrix whose $ (v, v) $-component is $ a_v $. 
		\item \label{item:prop:S4}
		For distinct vertices $ v, v' \in V_{\ge 2} $, the $ (v, v) $-component of $ S $ is in $ M_v \Z $ and the $ (v, v') $-component of $ S $ is in $ M_v M_{v'} \Z $. 
		\item \label{item:prop:S5}
		\[	
		S^{-1} \left(\Z^{V_{\ge 2}} \right) = \bigoplus_{v \in V_{\ge 2}} \frac{1}{M_v} \Z \subset \Q^{V_{\ge 2}}.
		\]		
	\end{enumerate}
\end{prop}

To prove this proposition, we prepare a lemma for the inverse matrices of symmetric block matrices.

\begin{lem} \label{lem:block_matrix}
	For a symmetric block matrix
	\[
	X = \pmat{A & B \\ {}^t\!B & C} \in \GL_{m+n} (\bbC)
	\]
	such that $ A \in \GL_m(\bbC) $ and $ C \in \GL_n(\bbC) $ be symmetric matrices and $ B \in \Mat_{m, n}(\bbC) $, let
	\[
	S \coloneqq (C - {}^t\!B A^{-1} B)^{-1} \in \GL_n(\bbC), \quad
	T \coloneqq -A^{-1} B \in \Mat_{m, n}(\bbC).
	\]
	Then, it holds that
	\[
	X^{-1} = \pmat{T \\ I} S \pmat{{}^t\!T & I} + \pmat{A^{-1} & \\ & O}, \quad
	\det S = \frac{\det A}{\det X}.
	\]	
	In particular, $ S $ is the $ n \times n $ bottom right submatrix of $ X^{-1} $. 
\end{lem}

\begin{proof}
	The first equality follows from
	\[
	X \left( \pmat{T \\ I} S \pmat{{}^t\!T & I} + \pmat{A^{-1} & \\ & O} \right) 
	= \pmat{ (AT+B) S {}^t\!T + I & (AT+B)S \\ ({}^t\!B T + C)S{}^t\!T + {}^t\!B A^{-1} & ({}^t\!B T + C)S }
	\]
	and the second equality follows from
	\[
	X = \pmat{ I &  \\  -{}^t\!T & I} \pmat{A & \\ &S^{-1}} \pmat{I & -T \\ & I}.
	\]
\end{proof}

Here we remark that \cref{lem:block_matrix} is a generalisation of the completing the square $ ax^2 + bx + c = a(x + b/2a)^2 - (b^2 - 4ac)/4a $ for symmetric block matrices.

\begin{proof}[Proof of $ \cref{prop:S} $]
	We obtain \cref{item:prop:S1} and the equalities
	\[
	S^{-1} = -W_{\ge 2} + W_{\ge 2, 1} W_1^{-1} W_{1, \ge 2}, \quad
	\det S = \frac{\det W_1}{\det W}
	\]
	by applying \cref{lem:block_matrix} for the block matrix
	\[
	W = \pmat{W_1 & W_{1, \ge 2} \\ W_{\ge 2, 1} & W_{\ge 2}} \in \GL_{m+n} (\Z).
	\]
	Since $ \det W \in \{ \pm 1 \}, W_1 = \diag(w_i)_{i \in V_1} $, and $ \det S > 0 $, we get \cref{item:prop:S2}. 
	\cref{item:prop:S3} follows from the fact that the $ (v, v') $-component of $ W_{\ge 2, 1} W_1^{-1} W_{1, \ge 2} $ is
	\[
	\sum_{i \in V_1, \, \{ i, v \}, \{ i, v' \} \in E} \frac{1}{w_i} = \delta_{v, v'} \sum_{i \in V_1, \, \{ i, v \} \in E} \frac{1}{w_i}.
	\]
	We prove \cref{item:prop:S4}. 
	By \cref{item:prop:S3}, the $ (v, v') $-cofactor of $ S^{-1} $ is an element of
	$ \left( \prod_{v^{\prime \prime} \in V_{\ge 2} \smallsetminus \{ v, v' \}} 1/M_{v^{\prime \prime}} \right) \Z $.
	Since $ \det S =  \pm \prod_{v \in V_{\ge 2}} M_v $ by \cref{item:prop:S3}, we obtain the claim.
	Finally, we prove \cref{item:prop:S5}.
	Let $ W'_{\ge 2} \coloneqq \diag(M_v)_{v \in V_{\ge 2}} \in \Aut(\Q^{\abs{V_{\ge 2}}}) $.
	It suffices to show $ S^{-1} W'_{\ge 2} \in \Aut(\Q^{\abs{V_{\ge 2}}}) $. 
	Since $ \det S = \pm \det W'_{\ge 2} $, we need only check that each component of $ S^{-1} W'_{\ge 2} $ is an integer, which follows from \cref{item:prop:S3}. 
\end{proof}


\subsection{Periodic maps determined by the linking matrix} \label{subsec:char}


In this subsection, we introduce periodic maps determined by the linking matrix and prove its basic properties.

For each vertex $ v \in V_{\ge 2} $, let 
\[
\calS_v
\coloneqq
\left\{ \frac{1}{2}\deg(v) - 1 + \sum_{i \in \overline{v}} \frac{l_i}{2w_i} \relmiddle{|} (l_i)_{i \in \overline{v}} \in \{ \pm 1 \}^{\overline{v}} \right\}
\]
Here we set $ \calS_v \coloneqq \left\{ \deg(v)/2 - 1 \right\} $ if $ \overline{v} = \emptyset $.
We also define the periodic map $ \veps_v \colon \frac{1}{2M_v}\Z/\Z \to \{ 0, \pm 1 \} $ by
\[
\veps_v(\alpha_v) \coloneqq 
\begin{dcases}
	\prod_{i \in \overline{v}} l_i 
	& \text{ if } \alpha_v \equiv \frac{1}{2}\deg(v) - 1 + \sum_{i \in \overline{v}} \frac{l_i}{2w_i}  \bmod \Z 
	\text{ for some } (l_i)_{i \in \overline{v}} \in \{ \pm 1 \}^{\overline{v}}, \\
	0 & \alpha_{v} \notin (\calS_v + \Z)/\Z.
\end{dcases}
\]
This map is well-defined by the following lemma.

\begin{lem} \label{lem:bij_calS_v}
	For a vertex $ v \in V_{\ge 2} $ such that $ \overline{v} \neq \emptyset $, the followings hold. 
	\begin{enumerate}
		\item \label{item:lem:bij_calS_v1} The map
		\[
		\begin{array}{ccc}
			\{ \pm 1 \}^{\overline{v}} & \longrightarrow & \calS_v \\
			(l_i)_{i \in \overline{v}} & \longmapsto & \displaystyle \frac{1}{2}\deg(v) - 1 + \sum_{i \in \overline{v}} \frac{l_i}{2w_i}
		\end{array}
		\]
		is bijective.
		\item \label{item:lem:bij_calS_v2}
		The natural projection $ \calS_v \to \frac{1}{2M_{v}} \Z/\Z $ is injective. 
		Thus, the map $ \veps_v $ is well-defined.
		\item \label{item:lem:bij_calS_v3}
		For any $ \alpha_v \in \calS_v $, it holds that $ \gcd(2M_v \alpha_v, M_v) = 1 $.
	\end{enumerate}
\end{lem}

\begin{proof}
	By \cref{lem:coprime}, there exists $ i_0 \in \overline{v} $ such that $ w_i $ is odd for any $ i \in \overline{v} \smallsetminus \{ i_0 \} $.
	Then, we have
	\begin{equation} 
		\Z/2M_v \Z 
		\cong 
		\Z/ 2w_{i_0} \Z \oplus \bigoplus_{i \in \overline{v} \smallsetminus \{ i_0 \}} \Z/w_i \Z
	\end{equation}
	by the Chinese remainder theorem.
	For each
	\[
	\alpha_v =
	\frac{1}{2}\deg(v) - 1 + \sum_{i \in \overline{v}} \frac{l_i}{2w_i}
	\in \calS_v, \quad
	(l_i)_{i \in \overline{v}} \in \{ \pm 1 \}^{\overline{v}},
	\]
	it holds that
	\[
	2 M_v \alpha_v
	\equiv
	\left\{
	\begin{alignedat}{3}
		&\frac{M_v}{w_i} l_i &
		&\bmod w_i &
		&\quad
		\text{ if }
		i \in \overline{v} \smallsetminus \{ i_0 \}, \\
		&M_v \deg(v) + \frac{M_v}{w_{i_0}} l_{i_0} 
		+ \sum_{i \in \overline{v} \smallsetminus \{ i_0 \} } \frac{M_v}{w_i} &
		&\bmod 2w_{i_0} &
		&\quad
		\text{ if }
		i = i_0.
	\end{alignedat}
	\right.
	\]
	Here $ M_v / w_i \not\equiv - M_v / w_i \bmod w_i $ for any vertex $ i \in \overline{v} \smallsetminus \{ i_0 \} $ since $ w_i $ is odd.
	We also have $ M_v / w_{i_0} \not\equiv - M_v / w_{i_0} \bmod w_{i_0} $ since we assume $ w_{i_0} \le -2 $ in \cref{subsec:graph}.
	Thus, we obtain \cref{item:lem:bij_calS_v2}. 
	By this argument, we can prove injectivity of the composition of the map $ \{ \pm 1 \}^{\overline{v}} \to \calS_v $ in \cref{item:lem:bij_calS_v1}
	and the map $ \calS_v \to \Z/2M_{v}\Z, \, \alpha_v \mapsto 2M_v \alpha_v $ which is injective by \cref{item:lem:bij_calS_v2}.
	Hence, we obtain \cref{item:lem:bij_calS_v1}.
	Moreover, we have \cref{item:lem:bij_calS_v3} since $ \gcd(M_v / w_i, w_i) = 1 $ for each $ i \in \overline{v} $. 
\end{proof}

In the above lemma, \cref{item:lem:bij_calS_v1} was first proved by Akihito Mori.
The proof presented here is due to the author.

The following lemma follows from the definition immediately.

\begin{lem} \label{lem:epsilon}
	For any vertex $ v \in V_{\ge 2} $, the followings hold.
	\begin{enumerate}
		\item \label{item:lem:epsilon2}
		\[
		\sum_{\alpha_v \in \frac{1}{2 M_v} \Z / \Z} \veps_v (\alpha_v) = 0.
		\]
		\item \label{item:lem:epsilon3}
		For each $ \alpha_v \in \frac{1}{2 M_v} \Z / \Z $ such that $ \veps_v (\alpha_v) \neq 0 $, $ M_v \alpha_v, M_v \alpha_v^2 \bmod \Z $ are independent of $ \alpha_v $ respectively.
		\item \label{item:lem:epsilon4}
		\[
		\sum_{\alpha_{v} \in \calS_{v}} \veps_{v}(\alpha_{v}) q^{\alpha_{v}}
		=
		q^{\deg(v)/2 - 1} \prod_{i \in \overline{v}} \left( q^{1 / 2w_{i}} - q^{-1 / 2w_{i}} \right).
		\]
	\end{enumerate}
\end{lem}

The following lemma is very important for our proof of our main theorem.
We use it in a proof of \cref{prop:Gauss_vanish} later.

\begin{lem}[{\cite[Lemma 8]{FIMT}}] \label{lem:FIMT}
	For $ 0 \le n \le \abs{\overline{v}} - 1 $, it holds that
	\[
	\sum_{\alpha_{v} \in \calS_{v}} \veps_{v}(\alpha_{v}) \alpha_v^n
	= 0.
	\]
\end{lem}

In the next subsection, we give a proof of this lemma by a different method from \cite{FIMT}.

At the end of this section, we consider products for vertices.
Let $ \calS \coloneqq \prod_{v \in V_{\ge 2}} \calS_v $ and the map $ \veps \colon (2S)^{-1}(\Z^{V_{\ge 2}}) / \Z^{V_{\ge 2}} \to \{ 0, \pm 1 \} $ be
\[
\veps \left((\alpha_v)_{v \in V_{\ge 2}} \right) \coloneqq \prod_{v \in V_{\ge 2}} \veps_v( \alpha_v).
\]
This is well-defined since $ \calS \subset (2S)^{-1}(\Z^{V_{\ge 2}}) $ by \cref{prop:S} \cref{item:prop:S5}.

The following lemma follows from \cref{lem:bij_calS_v}. 

\begin{lem} \label{lem:bij_calS}
	\begin{enumerate}
		\item \label{item:lem:bij_calS1} The map
		\[
		\begin{array}{ccc}
			\{ \pm 1 \}^{V_1} & \longrightarrow & \calS \\
			(l_i)_{i \in V_1} & \longmapsto & \displaystyle \left( \frac{1}{2}\deg(v) - 1 + \sum_{i \in \overline{v}} \frac{l_i}{2w_i} \right)_{v \in V_{\ge 2}}
		\end{array}
		\]
		is bijective.
		\item \label{item:lem:bij_calS1}
		The natural projection $ \calS \to (2S)^{-1}(\Z^{V_{\ge 2}}) / \Z^{V_{\ge 2}} $ is injective. 
	\end{enumerate}
\end{lem}


\subsection{Rational functions determined by the linking matrix} \label{subsec:rat_func}


In this subsection, we introduce rational functions determined by the linking matrix and prove its basic properties.

For a vertex $ v \in V_{\ge 2} $ and complex variable $ q $, define a rational function
\[
G_v (q)	\coloneqq (q^{M_{v}} - q^{-M_{v}})^{2 - \deg(v)} \prod_{i \in \overline{v}} \left( q^{M_{v} / w_{i}} - q^{-M_{v} / w_{i}} \right).
\]

It has the following Laurent expansion.

\begin{rem} \label{rem:G_v_symmetry}
	$ G_v(q^{-1}) = (-1)^{\deg (v) + \abs{\overline{v}}} G_v(q) $.
\end{rem}

The rational function $ G_v (q) $ has the following Laurent expansion.

\begin{lem} \label{lem:expansion_v}
	For a vertex $ v \in V_{\ge 2} $, $ G_{v} (q) $ is expanded as
	\[
	G_{v} (q)
	=
	\sum_{\alpha_{v} \in \calS_{v}} \veps_{v}(\alpha_{v}) \sum_{n_{v} = 0}^{\infty} \binom{n_{v} + \deg(v) - 3}{n_{v}} q^{2M_{v}(n_{v} + \alpha_{v})}
	\]
	for $ 0 < \abs{q}^{\sgn{M_v}} < 1 $.
	Here we define
	\begin{equation} \label{eq:binomial}
		\pmat{m \\ l}
		\coloneqq
		\begin{dcases}
			\frac{m!}{l! (m-l)!} & \text{ if } 0 \le l \le m, \\
			1 & \text{ if } m=-1, \\
			0 & \text{ otherwise}.
		\end{dcases}
	\end{equation}
\end{lem}

\begin{proof}
	The claim follows from \cref{lem:epsilon} \cref{item:lem:epsilon4} and the binomial theorem
	\[
	(1-q)^{-d} = \sum_{n=0}^{\infty} \pmat{n+d-1 \\n} q^n
	\]
	which holds for any integer $ d \ge 0 $.
\end{proof}

The rational function $ G_v (q) $ is expanded at $ q=1 $ as follows. 

\begin{lem} \label{lem:G_v_at_1}
	For a vertex $ v \in V_{\ge 2} $, it holds that
	\[
	G_v(q) = 2^{2 - \deg(v) + \abs{\overline{v}}} (M_v)^{1 - \deg(v) + \abs{\overline{v}}} (q-1)^{2 - \deg(v) + \abs{\overline{v}}}
	+O((q-1)^{3 - \deg(v) + \abs{\overline{v}}}).
	\]
\end{lem}

\begin{proof}
	The idea of our proof is due to Akihito Mori.
	The claim follows from
	\begin{align}
		G_v(q)
		&= q^{(\deg(v) - 1 - \sum_{i \in \overline{v}} 1/w_i) M_v} 
		(q^{2M_{v}} - 1)^{2 - \deg(v)} (q-1)^{\overline{v}}
		\prod_{i \in \overline{v}} \frac{q^{2M_{v} / w_{i}} - 1}{q-1} \\
		&= q^{(\deg(v) - 1 - \sum_{i \in \overline{v}} 1/w_i) M_v}
		(1 + q + \cdots + q^{2M_v - 1})^{2 - \deg(v)} (q-1)^{2 - \deg(v) + \overline{v}} 
		\prod_{i \in \overline{v}} \frac{q^{2M_{v} / w_{i}} - 1}{q-1}
	\end{align}
	and
	\[
	\lim_{q \to 1} \frac{q^{2M_{v} / w_{i}} - 1}{q-1}
	= \lim_{q \to 1} \left( 1 + q + \cdots + q^{2M_v / w_i - 1} \right)
	= \frac{2M_{v}}{w_i}.
	\]
\end{proof}

\cref{lem:FIMT} follows from \cref{lem:G_v_at_1}.
Thus, \cref{lem:G_v_at_1} is a refinement of \cref{lem:FIMT}. 

\begin{proof}[Proof of $ \cref{lem:FIMT} $]
	Since
	\[
	G_{v} (q)
	=
	(q^{2M_{v}} - 1)^{2 - \deg(v)}
	\sum_{\alpha_{v} \in \calS_{v}} \veps_{v}(\alpha_{v}) q^{2M_{v} \alpha_{v}},
	\]
	we have
	\[
	G_{v} (e^{-t})
	=
	\left( t^{2 - \deg(v)} + O(t^{3 - \deg(v)}) \right) 
	\sum_{n = 0}^{\infty} \frac{(-2M_v t)^{n}}{n!}
	\sum_{\alpha_{v} \in \calS_{v}} \veps_{v}(\alpha_{v}) \alpha_v^n.
	\]
	Since $ G_{v} (e^{-t}) = O(t^{2 - \deg(v) + \abs{\overline{v}}}) $ by \cref{lem:G_v_at_1}, we obtain the claim. 
\end{proof}

At the end of this section, we consider products for vertices.

For $ n \in \Z^{V_{\ge 3}} $, we denote
\begin{equation} \label{eq:prod_vertices}
	P(n) \coloneqq \prod_{v \in V_{\ge 4}} \frac{(n_{v} + \deg(v) - 3)(n_{v} + \deg(v) - 4) \cdots (n_v +1)}{(\deg(v) - 3)!}.
\end{equation}
We remark that for $ n \in \Z_{\ge -1}^{V_{\ge 3}} $ it holds that
\[
P(n) = \prod_{v \in V_{\ge 4}} \binom{n_{v} + \deg(v) - 3}{n_{v}}
= \prod_{v \in V_{\ge 3}} \binom{n_{v} + \deg(v) - 3}{n_{v}}
\]
by our definition of binomial coefficients in \cref{eq:binomial}.

For each complex number $ z $, we denote $ \bm{e}(z) \coloneqq e^{2\pi\iu z} $.
We fix a positive integer $ k $ and denote $ \zeta_k \coloneqq e^{2\pi\iu/k} $.

Then, the following holds by \cref{lem:expansion_v}. 

\begin{lem} \label{lem:expansion}
	For $ \mu \in \Z^{V_{\ge 2}} $ and $ t \in \R^{V_{\ge 2}}_{>0} $, it holds that
	\[
	\prod_{v \in V_{\ge 2}} G_{v} \left( \zeta_{2k M_v}^{\mu_v} e^{-t_v/2M_v} \right) 
	=
	\sum_{\alpha \in \calS} \veps(\alpha)
	\sum_{n \in \Z_{\ge 0}^{V_{\ge 3}}} P(n) \bm{e} \left( \frac{1}{k} {}^t\!(n + \alpha) \mu \right) \exp \left( -{}^t\!(n + \alpha) t \right) .
	\]
\end{lem}


\section{Calculations of WRT invariants} \label{sec:WRT}


In this section, we calculate WRT invariants of the plumbed homology sphere $ M(\Gamma) $.
Our starting point is the following expression by \cite{GPPV}. 

\begin{prop}[{\cite[Equation A.12]{GPPV}}]
	\begin{align}
		\WRT_k(M(\Gamma))
		=
		\frac{\bm{e}(\abs{V}/8) \zeta_{4k}^{-\sum_{v \in V} (w_{v} + 3)}}
		{2 \sqrt{2k}^{\abs{V}} (\zeta_{2k} - \zeta_{2k}^{-1})}
		\sum_{\mu \in (\Z \smallsetminus k\Z)^V/2k\Z^V} \zeta_{4k}^{{}^t\!\mu W \mu}
		\prod_{v \in V} \left( \zeta_{2k}^{\mu_v} - \zeta_{2k}^{-\mu_v} \right)^{2-\deg(v)}.
	\end{align}
\end{prop}

We can calculate WRT invariants as follows.

\begin{prop} \label{prop:WRT_expression}
	\begin{align}
		\WRT_k(M(\Gamma))
		= \,
		&\frac{(-1)^{\abs{V_1}} \bm{e}(\abs{V_{\ge 2}}/8) \zeta_{4k}^{-\sum_{v \in V} (w_{v} + 3) - \sum_{i \in V_1} 1/w_i}}
		{2 \sqrt{2k}^{\abs{V_{\ge 2}}} (\zeta_{2k} - \zeta_{2k}^{-1}) \prod_{i \in V_1} \sqrt{\abs{w_i}}}
		\\
		&\sum_{\mu \in (\Z \smallsetminus k\Z)^{V_{\ge 2}}/2k S(\Z^{V_{\ge 2}})}
		\bm{e} \left( -\frac{1}{4k} {}^t\!\mu S^{-1} \mu \right)
		\prod_{v \in V_{\ge 2}} G_v \left( \zeta_{2kM_v}^{\mu_v} \right).
	\end{align}
\end{prop}

To prove this, we need the following property called ``reciprocity of Gauss sums.''

\begin{prop}[{\cite[Theorem 1]{DT}}] \label{prop:reciprocity}
	Let $ L $ be a lattice of finite rank $ n $ equipped with a non-degenerated symmetric $ \Z $-valued bilinear form $ \sprod{\cdot, \cdot} $.
	We write
	\[
	L' \coloneqq \{ y \in L \otimes \R \mid \sprod{x, y} \in \Z \text{ for all } x \in L \} 
	\]
	for the dual lattice.
	Let $ 0 < k \in \abs{L'/L} \Z, u \in \frac{1}{k} L $, 
	and $ h \colon L \otimes \R \to L \otimes \R $ be a self-adjoint automorphism such that $ h(L') \subset L' $ and $ \frac{k}{2} \sprod{y, h(y)} \in \Z $ for all $ y \in L' $.
	Let $ \sigma $ be the signature of the quadratic form $ \sprod{x, h(y)} $.
	Then it holds that
	\begin{align}
		&\sum_{x \in L/kL} \bm{e} \left( \frac{1}{2k} \sprod{x, h(x)} + \sprod{x, u} \right)
		= \,
		&\frac{\bm{e}(\sigma/8) k^{n/2}}{\sqrt{\abs{L'/L} \abs{\det h}}}
		\sum_{y \in L'/h(L')} \bm{e} \left( -\frac{k}{2} \sprod{y + u, h^{-1}(y + u)} \right).
	\end{align}
\end{prop}

\begin{proof}[Proof of $ \cref{prop:WRT_expression} $]
	Idea of our proof is the same as in \cite[Proposition 6.1]{MM} which deal with the case when $ \Gamma $ is the H-graph.
	However, our calculation is slightly different and clearer than it.
	To begin with, since $ \zeta_{2k}^{\mu_v} - \zeta_{2k}^{-\mu_v} = 0 $ for $ \mu_v \in k\Z $, we can write
	\begin{equation} \label{eq:WRT_Gauss_sum} 
		\begin{aligned}
			&\sum_{\mu \in (\Z \smallsetminus k\Z)^V/2k\Z^V} \zeta_{4k}^{{}^t\!\mu W \mu}
			\prod_{v \in V} \left( \zeta_{2k}^{\mu_v} - \zeta_{2k}^{-\mu_v} \right)^{2-\deg(v)} \\
			= \,
			&\sum_{\mu \in (\Z/ 2k \Z)^{V_1} \oplus \left( (\Z \smallsetminus k\Z) / 2k\Z \right)^{V_{ \ge 2 }}} 
			\zeta_{4k}^{{}^t\!\mu W \mu}
			\prod_{v \in V} \left( \zeta_{2k}^{\mu_v} - \zeta_{2k}^{-\mu_v} \right)^{2-\deg(v)}.
		\end{aligned}
	\end{equation}
	Since
	\begin{align}
		\prod_{i \in V_1} \left( \zeta_{2k}^{\mu_i} - \zeta_{2k}^{-\mu_i} \right)
		&= 
		\prod_{i \in V_1} \sum_{l_i \in \{ \pm 1 \} } l_i \zeta_{2k}^{l_i \mu_i}
	\end{align}
	and for $ \mu = (\mu_1, \mu_{\ge 2}) \in \Z^{V}, \mu_1 \in \Z^{V_1}, \mu_{\ge 2} \in \Z^{V_{\ge 2}} $ it holds that
	\[
	{}^t\!\mu W \mu
	= {}^t\!\mu_{\ge 2} W \mu_{\ge 2} + \sum_{v \in V_{\ge 2}} \sum_{i \in \overline{v}} \left( w_i \mu_i^2 + 2 \mu_v \mu_i \right),
	\]
	the right hand side of \cref{eq:WRT_Gauss_sum} can be written as
	\begin{align}
		\sum_{\mu_{\ge 2} \in \left( (\Z \smallsetminus k\Z) / 2k\Z \right)^{V_{ \ge 2 }}} 
		\zeta_{4k}^{{}^t\!\mu_{\ge 2} W \mu_{\ge 2}}
		\prod_{v \in V_{\ge 2}}
		&\left( \zeta_{2k}^{\mu_v} - \zeta_{2k}^{-\mu_v} \right)^{2-\deg(v)}
		\\
		&\prod_{i \in \overline{v}} \sum_{l_i \in \{ \pm 1 \} } l_i
		\sum_{\mu_i \in \Z/ 2k \Z}
		\zeta_{4k}^{w_i \mu_i^2 + 2 (\mu_v + l_i) \mu_i}.
	\end{align}
	Since the last sum for $ \mu_i $ is equal to
	\[	
	\frac{\bm{e}(-1/8) \sqrt{2k}}{\sqrt{\abs{w_i}}}
	\sum_{\mu_i \in \Z/ w_i \Z}
	\zeta_{4k w_i}^{ -\left( 2k \mu_i + \mu_v + l_i \right)^2 }
	\]
	by reciprocity of Gauss sums (\cref{prop:reciprocity}), the right hand side of \cref{eq:WRT_Gauss_sum} can be written as
	\begin{equation} \label{eq:WRT_Gauss_sum2}
		\begin{aligned}
			\frac{\bm{e}(-\abs{V_1}/8) \sqrt{2k}^{\abs{V_1}}}{\prod_{i \in V_1} \sqrt{\abs{w_i}}}
			\sum_{\mu_{\ge 2} \in \left( (\Z \smallsetminus k\Z) / 2k\Z \right)^{V_{ \ge 2 }}} 
			\zeta_{4k}^{{}^t\!\mu_{\ge 2} W \mu_{\ge 2}}
			\prod_{v \in V_{\ge 2}}
			&\left( \zeta_{2k}^{\mu_v} - \zeta_{2k}^{-\mu_v} \right)^{2-\deg(v)}
			\\
			&\prod_{i \in \overline{v}} \sum_{l_i \in \{ \pm 1 \} } l_i
			\sum_{\mu_i \in \Z/ w_i \Z}
			\zeta_{4k w_i}^{ -\left( 2k \mu_i + \mu_v + l_i \right)^2 }.
		\end{aligned}
	\end{equation}	
	Since $ \bigoplus_{i \in \overline{v}} \Z/ w_i \Z \cong \Z / M_v \Z $ by the Chinese remainder theorem and \cref{lem:coprime}, the last line in \cref{eq:WRT_Gauss_sum2} is equal to
	\[
	\sum_{\mu'_v \in \Z/ M_v \Z}
	\prod_{i \in \overline{v}}
	\sum_{l_i \in \{ \pm 1 \} } l_i
	\zeta_{4k w_i}^{ -\left( 2k \mu'_v + \mu_v + l_i \right)^2 }.
	\]
	By replacing $ 2k \mu'_v + \mu_v $ by $ \mu_v $ and using \cref{prop:S} \cref{item:prop:S5}, the sum for $ \mu_{\ge 2} $ in \cref{eq:WRT_Gauss_sum2} is written as
	\begin{align}
		&\sum_{\mu_{\ge 2} \in (\Z \smallsetminus k\Z)^{V_{ \ge 2 }} / 2k S \left( \Z^{V_{ \ge 2 }} \right)} 
		\zeta_{4k}^{{}^t\!\mu_{\ge 2} W \mu_{\ge 2}} 
		\prod_{v \in V_{\ge 2}}
		\left( \zeta_{2k}^{\mu_v} - \zeta_{2k}^{-\mu_v} \right)^{2-\deg(v)} 
		\prod_{i \in \overline{v}} \sum_{l_i \in \{ \pm 1 \} } l_i
		\zeta_{4k w_i}^{ -\left( \mu_v + l_i \right)^2 } 
		\\
		= \,
		&\sum_{\mu_{\ge 2} \in (\Z \smallsetminus k\Z)^{V_{ \ge 2 }} / 2k S \left( \Z^{V_{ \ge 2 }} \right)} 
		\bm{e} \left( 
		\frac{1}{4k} \left( 
		{}^t\!\mu_{\ge 2} W \mu_{\ge 2}
		- \sum_{v \in V_{\ge 2}} \sum_{i \in \overline{v}} \frac{1}{w_i} \mu_v^2
		\right) \right)
		\\
		&\prod_{v \in V_{\ge 2}}
		\left( \zeta_{2k}^{\mu_v} - \zeta_{2k}^{-\mu_v} \right)^{2-\deg(v)}
		\prod_{i \in \overline{v}} \sum_{l_i \in \{ \pm 1 \} } l_i
		\zeta_{4k w_i}^{ -2 l_i \mu_v - l_i^2 }
		\\
		= \,
		&\zeta_{4k}^{- \sum_{i \in V_1} 1/w_i}
		\sum_{\mu_{\ge 2} \in (\Z \smallsetminus k\Z)^{V_{ \ge 2 }} / 2k S \left( \Z^{V_{ \ge 2 }} \right)} 
		\bm{e} \left( 
		\frac{1}{4k} \left( 
		{}^t\!\mu_{\ge 2} W \mu_{\ge 2}
		- \sum_{v \in V_{\ge 2}} \sum_{i \in \overline{v}} \frac{1}{w_i} \mu_v^2
		\right) \right) \\
		&\prod_{v \in V_{\ge 2}} (-1)^{\abs{\overline{v}}} G_v \left( \zeta_{2kM_v}^{\mu_v} \right).
	\end{align}
	By \cref{prop:S} \cref{item:prop:S3}, it holds that
	\[
	{}^t\!\mu_{\ge 2} W \mu_{\ge 2}
	- \sum_{v \in V_{\ge 2}} \sum_{ i \in \overline{v} } \frac{1}{w_i} \mu_v^2
	=
	-{}^t\!\mu_{\ge 2} S^{-1} \mu_{\ge 2}.
	\]
	Thus, we obtain the claim.
\end{proof}


\section{Expressions of homological blocks as false theta functions} \label{sec:HB_false}


In this section, we represent the homological block of $ M(\Gamma) $ as a false theta function.

For the plumbed homology sphere $ M(\Gamma) $, the homological block $ \widehat{Z}_{\Gamma} (q) $ is defined as follows.

\begin{dfn}[{\cite[Subsection 3.4]{GPPV}}]
	The homological block of the plumbed homology sphere $ M(\Gamma) $ is defined as
	\[
	\widehat{Z}_{\Gamma} (q) 
	= 
	2^{-\abs{V}} q^{-\sum_{v \in V} (w_{v} + 3)/4} 
	\mathrm{p.v.} \int_{\abs{z_{v}}=1, v \in V} \Theta_{-W,\delta }(q; z)
	\prod_{v \in V} (z_{v} - 1/z_{v})^{2 - \deg(v)} \frac{dz_{v}}{2\pi\sqrt{-1}z_{v}},
	\]
	where $ \mathrm{p.v.} $ is the Cauchy principal value defined as
	\begin{align}
		\mathrm{p.v.} \int_{\abs{z} = 1}
		\coloneqq \lim_{\veps \to +0} \left( \int_{\abs{z} = 1+\veps} + \int_{\abs{z} = 1-\veps} \right),
	\end{align}
	$ \delta \coloneqq (\deg(v))_{v \in V} \in \Z^V $, and
	\[
	\Theta_{-W, \delta} (q; z) \coloneqq 
	\sum_{l \in 2\Z^V + \delta} q^{-{}^t\!l W^{-1} l/4} \prod_{v \in V} z_v^{l_v}
	\]
	is the theta function. 
\end{dfn}

Here we remark that our definition of the Cauchy principal value is half of the previous definition in \cite{GPPV} (\cite[p.55]{GPPV}).
However, our definition of the homological block $ \widehat{Z}_{\Gamma} (q) $ is the same as \cite{GPPV} since we multiply it by $ 2^{-\abs{V}} $.

Our main result in this section is the following.

\begin{prop} \label{prop:HB_false}
	It holds
	\begin{align}
		\widehat{Z}_{\Gamma} (q) 
		= \,
		&(-1)^{\abs{V_1}} 2^{-\abs{V_{\ge 3}}} q^{-(\sum_{v \in V} (w_{v} + 3) + \sum_{i \in V_1} 1/w_i)/4} \\
		&\sum_{e \in \{ \pm 1 \}^{V_{\ge 3}}} \left( \prod_{v \in V_{\ge 3}} e_v^{\deg (v) + \abs{\overline{v}}} \right)
		\sum_{\alpha \in \calS} \veps(\alpha)
		\sum_{n \in \Z^{V_{\ge 3}}} P(n) q^{Q(e(n + \alpha))},
	\end{align}
	where $ P(n) $ is the polynomial defined in \cref{eq:prod_vertices},
	$ Q(x) \coloneqq {}^t\!x S x $ and $ ex \coloneqq ((x_v)_{v \in V_{2}}, (e_v x_v)_{v \in V_{\ge 3}}) $ 
	for $ x \in \Z^{V_{\ge 2}} $ and $ e \in \{ \pm 1 \}^{V_{\ge 3}} $. 
\end{prop}

To prove this proposition, we need the following lemma. 

\begin{lem}[{\cite[p. 743]{Andersen-Mistegard}}] \label{lem:Cauchy_principal_value}
	For $ d \in \Z_{>0} $ and $ l \in \Z $, it holds that
	\[
	\mathrm{p.v.} \int_{\abs{z}=1} \frac{z^{l-1}}{(z - 1/z)^{d - 2}} \frac{dz}{2\pi\sqrt{-1}}
	=
	\begin{dcases}
		-2l & d=1, l \in \{ \pm 1 \}, \\
		2 & d=2, l=0, \\
		\sgn(l)^d \binom{m+d-3}{d-3} & d \ge 3, \pm l \in d - 2 + 2m \text{ for some } m \in \Z_{\ge 0}, \\
		0 & \text{ otherwise.}
	\end{dcases}
	\]
\end{lem}

\begin{proof}[Proof of $ \cref{prop:HB_false} $]
	By definition, we have
	\begin{align}
		&2^{\abs{V}} q^{\sum_{v \in V} (w_{v} + 3)/4}
		\widehat{Z}_{\Gamma} (q) \\
		= \,
		&\sum_{l \in 2\Z^V + \delta} q^{-{}^t\!l W^{-1} l/4} 
		\prod_{v \in V}
		\mathrm{p.v.} \int_{\abs{z_v}=1}
		\frac{ z_v^{l_v - 1} }{(z_v - 1/z_v)^{\deg(v) - 2}} \frac{dz_v}{2\pi\sqrt{-1}}.
	\end{align}
	By \cref{lem:Cauchy_principal_value}, this is equal to
	\begin{align}
		\sum_{\substack{
				l = (l_1, l_2, l_{\ge 3}), \\
				l_1 \in \{ \pm 1 \}^{V_1}, \,
				l_2 = 0, \\
				\pm l_{\ge 3} \in 2\Z_{\ge -1}^{V_{\ge 3}} + \delta_{\ge 3}
		}}
		&q^{-{}^t\!l W^{-1} l/4} 
		\left( \prod_{i \in V_1} (-2l_i) \right)  2^{\abs{V_2}} \\
		&\prod_{v \in V_{\ge 3}} \sgn(l_v)^{\deg (v)} \binom{(\abs{l_v} - \deg(v) + 2)/2 + \deg(v)-3}{\deg(v)-3},
	\end{align}
	where $ \delta_{\ge 3} \coloneqq (\deg(v))_{v \in V_{\ge 3}} \in \Z^{V_{\ge 3}} $.
	For $ l = {}^t\!(\varepsilon_1, l_2, l_{\ge 3}) $ with $ l_2 = 0 $, we have
	\[
	-{}^t\!l W^{-1} l
	= Q( l_{\ge 3} + {}^t\!T l_1) - \sum_{i \in V_1} \frac{l_i^2}{w_i}
	\]
	by \cref{prop:S} \cref{item:prop:S1}.
	By letting
	\[
	l_i = e_v \varepsilon_i, \quad
	l_v = e_v (n_v + \deg(v) - 2), \quad
	e_v, \varepsilon_i \in \{ \pm 1 \}, \quad
	n_v \in \Z_{\ge 0}
	\]
	for $ v \in V_{\ge 3} $ and $ i \in \overline{v} $, we obtain
	\begin{align}
		&(-1)^{\abs{V_1}} 2^{\abs{V_{\ge 3}}} q^{(\sum_{v \in V} (w_{v} + 3) + \sum_{i \in V_1} 1/w_i)/4}
		\widehat{Z}_{\Gamma} (q) \\
		= \,
		&\sum_{e \in \{ \pm 1 \}^{V_{\ge 3}}} \left( \prod_{v \in V_{\ge 3}} e_v^{\deg (v)} \right)
		\sum_{\varepsilon \in \{ \pm 1 \}^{V_1}}
		\left( \prod_{v \in V_{\ge 3}} \prod_{i \in \overline{v}} e_v \varepsilon_i \right) \\
		&\sum_{n \in \Z_{\ge 0}^{V_{\ge 3}}}
		q^{Q( e(n + \delta_{\ge 3}/2 + {}^t\!T \varepsilon/2 - (1, \dots, 1)) )}
		\prod_{v \in V_{\ge 3}} \binom{n_{v} + \deg(v) - 3}{n_{v}}.
	\end{align}
	Since $ W^{-1} \in \Aut(\Z^V) $, every matrix entries of $ S {}^t\!T $ is integers by \cref{prop:S} \cref{item:prop:S1}. 
	Thus, for any $ \varepsilon \in \{ \pm 1 \}^{V_1} $, we have $ {}^t\!T \varepsilon/2 \in (2S)^{-1}(\Z^{V_{\ge 2}}) $. 
	Since 
	\[
	\frac{1}{2} {}^t\!T \varepsilon = \left( \sum_{i \in \overline{v}} \frac{\varepsilon_i}{2w_i} \right)_{v \in V_{\ge 2}}
	\]
	by the definition of $ T $ in \cref{subsec:S}, the map
	\[
	\begin{array}{ccc}
		\{ \pm 1 \}^{V_1} & \longrightarrow & \calS \\
		\varepsilon & \longmapsto & \delta_{\ge 3}/2 + {}^t\!T \varepsilon/2 - (1, \dots, 1) 
	\end{array}
	\]
	is bijective by \cref{lem:bij_calS} \cref{item:lem:bij_calS1}.
	Thus, the claim follows by letting $ \alpha = \delta_{\ge 3}/2 + {}^t\!T \varepsilon/2 - (1, \dots, 1) $.
\end{proof}


\section{An asymtotic formula} \label{sec:asymp}


In this section, we develop an asymptotic formula which we need to calculate radial limits of homological blocks.


\subsection{Asymtotic formulas obtained from Euler--Maclaurin summation formula} \label{subsec:Euler--Maclaurin}


Some useful methods have been developed to calculate radial limits of false theta functions.
The method of Lawrence--Zagier~\cite[p. 98, Proposition]{LZ} is the beginning of these.
Lawrence--Zagier~\cite{LZ} developed it by using $ L $-functions for periodic maps.
Zagier~\cite{Zagier_asymptotic} collected many techniques to calculate asymptotic expansions of infinite series.
In particular, \cite[Equation (44)]{Zagier_asymptotic} is the very powerful method followed by the Euler--Maclaurin summation formula.
In our setting, these methods deal with the case when $ \abs{V_{\ge 2}} = 1 $ (in this case, $ M(\Gamma) $ is a Seifert $ 3 $-manifold).
Bringmann--Kaszian--Milas~\cite[Equation (2.8)]{BKM} and Bringmann--Mahlburg--Milas~\cite[Lemma 2.2]{BMM_high_depth} stated the two-variable case of \cite[Equation (44)]{Zagier_asymptotic}.

In this subsection, we develop a more general formula (\cref{prop:Euler--Maclaurin_poly}) to deal with $ n $-variable case with polynomial weights.

To begin with, we prepare the notation for asymptotic expansion by Poincar\'{e}. 

\begin{dfn}[Poincar\'{e}]
	Let $ L $ be a positive number, $ f \colon \R_{>0} \to \bbC $ be maps, $ t $ be a variable of $ \R_{>0} $, and $ (a_n)_{n \ge -L} $ be a family of complex numbers.
	Then, we write
	\[
	f(t) \sim \sum_{n \ge -L} a_n t^{n} \text{ as } t \to +0
	\]
	if for any positive number $ M $ there exist positive numbers $ K_M $ and $ \varepsilon $ such that
	\[
	\abs{ f(t) - \sum_{-L \le n \le M} a_n t^{n} }
	\le K_M \abs{ t^{M+1}}
	\]
	for any $ 0 < t < \varepsilon $.
	In this case, we call the infinite series $ \sum_{n \ge -L} a_n t^{n} $ as the \textbf{asymptotic expansion} of $ f(t) $ as $ t \to +0 $. 
\end{dfn}

Here we remark that asymptotic expansions are typically divergent series.

We also need the following terminology.

\begin{dfn}
	A function $ f \colon \R \to \bbC $ is called of \textbf{rapid decay} as $ x \to \infty $ if $ x^m f^{(n)} (x) $ is bounded as $ x \to \infty $ for any $ m, n \in \Z_{\ge 0} $.
\end{dfn}

Our starting point is the following lemma followed by the Euler--Maclaurin summation formula.

\begin{lem}[{\cite[Equation (44)]{Zagier_asymptotic}}] \label{lem:Euler--Maclaurin}
	Let $ N $ be a positive number and $ f \colon \R^N \to \bbC $ be a $ C^\infty $ function of rapid decay as $ x_1, \dots, x_N \to \infty $.
	Fix $ \alpha \in \R^N $.
	Then, for a variable $ t \in \R_{>0} $, an asymptotic expansion as $ t \to +0 $
	\[
	\sum_{n \in \Z_{\ge 0}^N} f(t(n+\alpha))
	\sim \sum_{-1 \le n_i, \, 1 \le i \le N} \left( \prod_{1 \le i \le N} \frac{B_{n_i+1}(\alpha_i)}{(n_i+1)!} \right)
	f^{(n)}(0) t^{n_1 + \cdots + n_N}
	\]
	holds.
	Here, $ F(t) \sim G(t) $ means that $ F(t) = G(t) + O(t^R) $ for any positive number $ R $,
	$ B_i(x) $ is the $ i $-th Bernoulli polynomial, and
	\[
	g^{(-1)}(x) = \frac{d^{-1}}{d x^{-1}} g(x) \coloneqq -\int_x^\infty g(x') dx', \quad
	f^{(n)}(x) \coloneqq \frac{\partial^{n_1 + \cdots + n_N} f}{\partial x_1^{n_1} \cdots \partial x_r^{n_N}} (x).
	\]
\end{lem}

Here we remark that Zagier~\cite[Equation (44)]{Zagier_asymptotic} stated for the case $ N=1 $ and Bringmann--Kaszian--Milas~\cite[Equation (2.8)]{BKM} and Bringmann--Mahlburg--Milas~\cite[Lemma 2.2]{BMM_high_depth} stated for the case $ N=2 $.

By \cref{lem:Euler--Maclaurin}, we obtain the following asymptotic expansion formula.

\begin{prop} \label{prop:Euler--Maclaurin_poly}
	Let $ N $ and $ N' $ be non-negative integers, $ f \colon \R^{N+N'} \to \bbC $ be a $ C^\infty $ function of rapid decay as $ x_1, \dots, x_N \to \infty $, and $ P(x) = \sum_{m \in \Z_{\ge 0}^N} p_m x_1^{m_1} \cdots x_N^{m_N} $ be a polynomial.
	Fix $ \alpha, \lambda \in \R^N $ and $ \alpha' \in \R^{N'} $.
	Then, for a variable $ t \in \R_{>0} $, an asymptotic expansion as $ t \to +0 $
	\[
	\sum_{n \in \Z_{\ge 0}^N} P(\lambda + n) f(t(\alpha + \lambda + n), t\alpha')
	\sim \sum_{n \in \Z^N \times \Z_{\ge 0}^{N'}} t^{n_1 + \cdots + n_{N+N'}} f^{(n)}(0) 
	\frac{{\alpha_1'}^{n_{N+1}} \cdots {\alpha_{N'}'}^{n_{N+N'}}}{n_{N+1}! \cdots n_{N+N'}!} 
	\sum_{m \in \Z_{\ge 0}^N} p_m \bbB_{m, n}(\alpha, \lambda)
	\]
	holds.
	Here we define
	\begin{align}
		\bbB_{m, n}(\alpha, \lambda)
		&\coloneqq
		\prod_{1 \le i \le N} \bbB_{m_i, n_i} (\alpha_i, \lambda_i), \\
		\bbB_{m_i, n_i} (\alpha_i, \lambda_i)
		&\coloneqq
		\begin{dcases}
			\sum_{0 \le l \le m_i + n_i + 1} b_{m_i, n_i, l} B_{m_i + n_i + 1 - l}(\lambda_i) \alpha_i^l & \text{ if } n_i \ge 0, \\
			\frac{m_i!}{(m_i + n_i + 1)!} (-\alpha_i)^{m_i + n_i + 1} & \text{ if } -m_i - 1 \le n_i \le -1, \\
			0 & \text{ if } n_i \le -m_i - 2,
		\end{dcases}
		\\
		b_{m_i, n_i, l} 
		&\coloneqq 
		\frac{m_i!}{(m_i + n_i + 1- l)!} \sum_{0 \le k \le l} \pmat{m_i + n_i - k \\ n_i} \frac{(-1)^k}{k! (l-k)!}.
	\end{align}
\end{prop}

\begin{proof}
	It suffices to show the case when $ (N, N') = (1, 0) $ or $ (0,1) $. 
	In the last case, the claim follows from the Taylor's formula.
	In the following, we assume $ (N, N') = (1, 0) $.
	Let $ g_\alpha(m, f; x) \coloneqq (x-\alpha)^m f(x) $.
	By \cref{lem:Euler--Maclaurin}, we have
	\begin{align}
		\sum_{n \ge 0} P(\lambda + n) f(t(\alpha + \lambda + n))
		&= \sum_{m \ge 0} p_m t^{-m} \sum_{n \ge 0} \left( t(\alpha + \lambda + n) - t \alpha \right)^m f(t(\alpha + \lambda + n)) \\
		&= \sum_{m \ge 0} p_m t^{-m} \sum_{n \ge 0} g_{t \alpha}(m, f; t(\alpha + \lambda + n)) \\
		&\sim \sum_{m \ge 0} p_m t^{-m} \sum_{n \ge -1} \frac{B_{n+1}(\alpha + \lambda)}{(n+1)!} g_{t \alpha}^{(n)} (m, f; 0) t^n \\
		&= \sum_{n > - \infty} t^n \sum_{m \ge 0} p_m \frac{B_{m+n+1}(\alpha + \lambda)}{(m+n+1)!} g_{t \alpha}^{(m+n)} (m, f; 0).
	\end{align}
	We need the following lemma.
	
	\begin{lem} \label{lem:g_diff}
		For any $ n \ge -1 $, it holds that
		\begin{align}
			g_{\alpha}^{(n)}(m, f; 0)
			&=
			\sum_{0 \le l \le m} \pmat{m \\ l} \pmat{n \\ l} l! (-\alpha)^{m - l} f^{(n-l)} (0) \\
			&=
			m! \sum_{0 \le l \le m} \pmat{n \\ l} \frac{(-\alpha)^{m-l}}{(m-l)!} f^{(n-l)} (0).
		\end{align}
	\end{lem}
	
	\begin{proof}
		We recall that $ \binom{n}{l} = 1 $ for $ n=-1 $ by our definition of binomial coefficients in \cref{eq:binomial}.
		If $ n \ge 0 $, then by the Leibnitz rule, we have
		\[
		g_{\alpha}^{(n)} (m, f; 0)
		=
		\sum_{0 \le l \le n} \pmat{n \\ l} \restrict{\frac{d^l}{dx^l} (x - \alpha)^m}{x=0} f^{(n-l)}(0).
		\]
		Thus, we obtain the claim.
		If $ n = -1 $, then the claim follows by induction and
		\[
		g_{\alpha}^{(-1)} (m, f; 0)
		= (-\alpha)^m f^{(-1)} (0) + m g_{\alpha}^{(-1)} \left(m-1, f^{(-1)}; 0 \right)
		\]
		which follows from an integration by parts.
	\end{proof}
	
	We turn back to the proof of \cref{prop:Euler--Maclaurin_poly}. 
	By \cref{lem:g_diff}, we have
	\begin{align}
		&\sum_{n \ge 0} P(\lambda + n) f(t(\alpha + \lambda + n)) \\
		\sim 
		&\sum_{n > - \infty} t^n \sum_{m \ge 0} p_m \frac{B_{m+n+1}(\alpha + \lambda)}{(m+n+1)!}
		m! \sum_{0 \le l \le m} \pmat{m+n \\ l} \frac{(-t\alpha)^{m-l}}{(m-l)!} f^{(m+n-l)} (0).
	\end{align}
	By replacing $ m+n-l $ by $ n $, this can be written as
	\[
	\sum_{n > - \infty} t^n f^{(n)} (0) \sum_{m \ge 0} p_m 
	m! \sum_{0 \le l \le m} \frac{B_{n+l+1}(\alpha + \lambda)}{(n+l+1)!} \pmat{n+l \\ l} \frac{(-\alpha)^{m-l}}{(m-l)!}.
	\]
	Thus, it suffices to show that
	\begin{equation} \label{eq:bbB}
		\bbB_{m, n}(\alpha, \lambda)
		= 
		m! \sum_{0 \le l \le m} \frac{B_{n+l+1}(\alpha + \lambda)}{(n+l+1)!} \pmat{n+l \\ l} \frac{(-\alpha)^{m-l}}{(m-l)!}.
	\end{equation}
	By the addition formula of Bernoulli polynomials (for example, see \cite[Equation 23.1.7]{Handbook_formulas})
	\[
	B_i(\alpha + \lambda) = \sum_{0 \le j \le i} \pmat{i \\ j} B_j(\lambda) \alpha^j,
	\]
	the right hand side of \cref{eq:bbB} is equal to
	\[
	m! \sum_{0 \le l \le m} \pmat{n+l \\ l} \frac{(-\alpha)^{m-l}}{(m-l)!}
	\sum_{0 \le k \le n+l+1} \frac{B_{n+l+1-k}(\lambda)}{(n+l+1-k)!} \frac{\alpha^k}{k!}.
	\]
	By letting $ i \coloneqq m+k-l, j \coloneqq m-l $, this is written as
	\[
	\sum_{0 \le i \le m+n+1} B_{m+n+1-i}(\lambda) \alpha^i
	\frac{m!}{(m+n+1-i)!} \sum_{0 \le j \le i} \pmat{m+n-j \\ n} \frac{(-1)^{j}}{j! (i-j)!}.
	\]
	Thus, we obtain \cref{eq:bbB} for the case when $ n \ge 0 $. 
	Assume $ n \le -1 $. 
	Then, the right hand side of \cref{eq:bbB} is $ 0 $ unless $ m+n-j = -1 $ for some $ 0 \le j \le i \le m+n+1 $. 
	It occurs only for the case when $ i = j = m+n+1 $. 
	Thus, if $ m+n+1 < 0 $ then the right hand side of \cref{eq:bbB} is $ 0 $ and \cref{eq:bbB} holds. 
	If $ m+n+1 \ge 0 $, the right hand side of \cref{eq:bbB} is written as
	\[
	\frac{m!}{(m+n+1)!} (-\alpha)^{m+n+1}
	\]
	and thus, \cref{eq:bbB} also holds in this case. 
\end{proof}

\begin{rem}
	If $ m_i = 0 $, then since $ b_{0, n_i, l} = 1/l! (n_i + 1 - l)! $ we have
	\[
	\bbB_{0, n_i} (\alpha_i, \lambda_i) = B_{n_i + 1} (\lambda_i + \alpha_i)
	\]
	by the addition formula of Bernoulli polynomials (for example, see \cite[Equation 23.1.7]{Handbook_formulas}).
	Thus, in this case, \cref{prop:Euler--Maclaurin_poly} coincides with \cite[Equation (44)]{Zagier_asymptotic}, \cite[Equation (2.8)]{BKM}, and \cite[Lemma 2.2]{BMM_high_depth}.
\end{rem}


\subsection{Asymtotic expansion of infinite series with weighted Gauss sums} \label{subsec:asymp_Gauss}


In this subsection, we study a family of infinite series $ F(f; t) $ with weighted Gauss sums, which involves both WRT invariants and homological blocks.
We start with the following asymptotic formula following from \cref{prop:Euler--Maclaurin_poly}.
In the following statement, we use notations $ \varepsilon, \calS, Q $ and so on, which we prepared in \cref{sec:fund_data,sec:HB_false}.

\begin{cor} \label{cor:asymp_F(f; t)}
	For $ e \in \{ \pm 1 \}^{V_{\ge 3}} $ and a $ C^\infty $ function $ f \colon \R^{V_{\ge 2}} \to \bbC $ of rapid decay as $ x_v \to \infty $ for each $ v \in V_{\ge 2} $, let
	\[
	F_e(f; t) \coloneqq 
	\sum_{\alpha \in \calS} \veps(\alpha)
	\sum_{n \in \Z_{\ge 0}^{V_{\ge 3}}}
	\bm{e}\left( \frac{1}{k} Q(e(n + \alpha)) \right) P(n) f(t(n + \alpha))
	\]
	be a series for a variable $ t \in \R_{>0} $.
	Then, it holds that
	\begin{align} \label{eq:asymp_F(f; t)}
		F_e(f; t) 
		\sim
		&\sum_{n \in \Z^{V_{\ge 2}}} c_{e, n} f^{(n)}(0) t^{\sum_{v \in V_{\ge 2}} n_v}
	\end{align}
	as $ t \to +0 $, where we define
	\begin{align}
		c_{e, n} = \,
		&\sum_{\alpha \in \calS} \veps(\alpha)
		\left( \prod_{v \in V_2} \frac{\alpha_v^{n_v}}{n_v!} \right)
		\sum_{m \in \Z_{\ge 0}^{V_{\ge 3}}} p_m \left( \prod_{v \in V_{\ge 3}} k^{m_v + n_v} \right)
		\sum_{\lambda \in \{ 0, \dots, k-1 \}^{V_{\ge 3}}}
		\bm{e}\left( \frac{1}{k} Q(e(\lambda + \alpha)) \right) \bbB_{m, n}(\alpha, \lambda)
	\end{align}
	for $ n \in \Z^{V_{\ge 2}} $.
\end{cor}

We define two $ C^\infty $ function of rapid decay $ f_1, f_2 \colon \R^{V_{\ge 2}} \to \bbC $ 
as $ f_1(x) \coloneqq \exp(-\sum_{v \in V_{\ge 2}} x_v) $ and $ f_2 \coloneqq \exp(-Q(x)) $ respectively.
Then, WRT invariants and limit values of homological blocks can be written as limit values of $ F(f_1; t) $ and $ F(f_2; t) $ as $ t \to 0 $ shown later in \cref{prop:WRT_F(f_1;t),prop:HB_lim}, respectively.

Here, we need to prove that $ F(f; t) $ has a limit as $ t \to 0 $, that is, $ c_n $ vanishes for $ n \in \Z^{V_{\ge 3}} $ with non-positive components.
Such a result is called ``a vanishing result of weighted Gauss sums'' in \cite{MM}, and it is the most important part in a proof of the conjecture of Gukov--Pei--Putrov--Vafa~\cite[Conjecture 2.1, Equation (A.28)]{GPPV}.
In the previous works~\cite{BMM_high_depth,MM}, it is proved by direct calculations (\cite[Theorem 4.1]{BMM_high_depth},\cite[Proposition 4.2]{MM}).
On the other hand, our proof is based on the above asymptotic expansion in \cref{cor:asymp_F(f; t)}.
Since $ c_n $ is independent of $ f $ in this asymptotic expansion, it suffices to consider not for any $ f $, but for $ f_1 $.
Since $ F(f_1; t) $ is a rational function shown later in \cref{lem:F(f_1;t)}, vanishings results of $ c_n $ is equivalent to holomorphy of $ F(f_1; t) $, which is proved in \cref{sec:proof_main}.

The homological block $ \widehat{Z}_{\Gamma} (\zeta_k e^{-t^2}) $ can be written in term of $ F(f_2; t) $ as follows.

\begin{prop} \label{prop:HB_lim}
	It holds that
	\[
	\restrict{\left(
		(-1)^{\abs{V_1}} q^{-(\sum_{v \in V} (w_{v} + 3) + \sum_{i \in V_1} 1/w_i)/4} 
		\widehat{Z}_{\Gamma} (q) 
		\right)}{q =\zeta_k e^{-t^2}}
	= 
	\sum_{e \in \{ \pm 1 \}^{V_{\ge 3}}} \left( \prod_{v \in V_{\ge 3}} e_v^{\deg (v) + \abs{\overline{v}}} \right)
	F_e(f_2; t).
	\]
\end{prop}

This proposition follows from \cref{prop:HB_false} for $ C(n) = \bm{e}\left( Q(e(n + \alpha))/k \right) $.

In the last of this subsection, we prove that $ F(f_1; t) $ is a rational function related to WRT invariants.
To begin with, we replace quadratic forms $ Q(n + \alpha) $ as linear forms in the indices of $ \zeta_k $ in the definition of $ F(f; t) $.

\begin{lem} \label{lem:F(f;t)}
	For $ e \in \{ \pm 1 \}^{V_{\ge 3}} $ and a $ C^\infty $ function $ f \colon \R^{V_{\ge 2}} \to \bbC $ of rapid decay as $ x_v \to \infty $ for each $ v \in V_{\ge 2} $, it holds that
	\begin{align}
		F_e(f; t) = 
		&\frac{\bm{e} \left( \abs{V_{\ge 2}}/8 \right)}{\sqrt{2k}^{\abs{V_{\ge 2}}} \prod_{i \in V_1} \sqrt{\abs{w_i}}}
		\sum_{ \mu \in \Z^{V_{\ge 2}} / 2kS \left(\Z^{V_{\ge 2}} \right) } 
		\bm{e} \left( -\frac{1}{4k} {}^t\!\mu S^{-1} \mu \right) \\
		&\sum_{\alpha \in \calS} \veps(\alpha)
		\sum_{n \in \Z_{\ge 0}^{V_{\ge 3}}}
		\bm{e}\left( \frac{1}{k} {}^t\!\mu e(n + \alpha) \right) P(n) f(t(n + \alpha)).
	\end{align}
\end{lem}

\begin{proof}
	Let
	\[
	G(2kS) \coloneqq
	\sum_{ \mu \in \Z^{V_{\ge 2}} / 2kS \left(\Z^{V_{\ge 2}} \right) } 
	\bm{e} \left( -\frac{1}{4k} {}^t\!\mu S^{-1} \mu \right).
	\]
	Then, we have $ G(2kS) = \left( \sqrt{2k} \bm{e} \left( -1/8 \right) \right)^{\abs{V_{\ge 2}}} \sqrt{\det S} $ by \cref{prop:reciprocity}. 
	Here we remark $ \det S = \prod_{i \in V_1} \abs{w_i} $ by \cref{prop:S} \cref{item:prop:S2}.
	Then, we can write
	\begin{align}
		&G(2kS) F_e(f; t)  \\
		= \,
		&\sum_{ \mu \in \Z^{V_{\ge 2}} / 2kS \left(\Z^{V_{\ge 2}} \right) } 
		\sum_{\alpha \in \calS} \veps(\alpha)
		\sum_{n \in \Z_{\ge 0}^{V_{\ge 3}}}
		\bm{e}\left( \frac{1}{k} \left( Q(e(n + \alpha)) - \frac{1}{4} {}^t\!\mu S^{-1} \mu \right) \right) P(n) f(t(n + \alpha)).
	\end{align}
	By replacing $ \mu $ by $ \mu - 2Se(n + \alpha) $, we obtain the claim.
\end{proof}

By this lemma, we obtain the following representation.

\begin{lem} \label{lem:F(f_1;t)}
	For $ e \in \{ \pm 1 \}^{V_{\ge 3}} $,
	\begin{align}
		F_e(f_1; t) = 
		e_v^{\deg (v) + \abs{\overline{v}}}
		\frac{\bm{e} \left( \abs{V_{\ge 2}}/8 \right)}{(2k)^{\abs{V_{\ge 2}}} \prod_{i \in V_1} \sqrt{\abs{w_i}}}
		\sum_{ \mu \in \Z^{V_{\ge 2}} / 2kS \left(\Z^{V_{\ge 2}} \right) } 
		\bm{e} \left( -\frac{1}{4k} {}^t\!\mu S^{-1} \mu \right) 
		\prod_{v \in V_{\ge 2}} G_v \left( \zeta_{2kM_v}^{\mu_v} e^{-e_v t/2M_v} \right).
	\end{align}
	In particular, $ F_e(f_1; t) $ is a meromorphic function.
\end{lem}

\begin{proof}
	By \cref{lem:F(f;t)}, we have
	\begin{align}
		G(2kS) F_e(f_1; t)
		= \,
		&\sum_{ \mu \in \Z^{V_{\ge 2}} / 2kS \left(\Z^{V_{\ge 2}} \right) } 
		\bm{e} \left( -\frac{1}{4k} {}^t\!\mu S^{-1} \mu \right) \\
		&\sum_{\alpha \in \calS} \veps(\alpha)
		\sum_{n \in \Z_{\ge 0}^{V_{\ge 3}}}
		\bm{e}\left( \frac{1}{k} {}^t\!\mu e(n + \alpha) \right) P(n) \exp \left( - \sum_{ v \in V_{\ge 3} } t_v (n_v + \alpha_{v}) \right).
	\end{align}
	By applying \cref{lem:expansion} for this, we have
	\[
	G(2kS) F_e(f_1; t)
	= 
	\sum_{ \mu \in \Z^{V_{\ge 2}} / 2kS \left(\Z^{V_{\ge 2}} \right) } 
	\bm{e} \left( -\frac{1}{4k} {}^t\!\mu S^{-1} \mu \right)
	\prod_{v \in V_{\ge 2}} G_v \left( \zeta_{2kM_v}^{e_v \mu_v} e^{-t/2M_v} \right).
	\]
	By \cref{rem:G_v_symmetry}, we obtain the claim. 
\end{proof}


\section{A proof of the main theorem} \label{sec:proof_main}


Let 
\[
F(\bm{t}) \coloneqq
\sum_{ \mu \in \Z^{V_{\ge 2}} / 2kS \left(\Z^{V_{\ge 2}} \right) } 
\bm{e} \left( -\frac{1}{4k} {}^t\!\mu S^{-1} \mu \right) 
\prod_{v \in V_{\ge 2}} G_v \left( \zeta_{2kM_v}^{\mu_v} e^{-t_v/2M_v} \right).
\]
In this section, we prove holomorphy of this meromorphic function at $ \bm{t} = 0 $ (that is, vanishings of $ c_n $) and our main result.


\subsection{The main result} \label{subsec:proof_main}


In this subsection, we state our main result without a proof.

\begin{prop} \label{prop:Gauss_vanish}
	\begin{enumerate}
		\item \label{item:prop:Gauss_vanish1}
		For each vertex $ v \in V_{\ge 3} $, the order of $ F(\bm{t}) $ at $ t_v=0 $ satisfies
		\[
		\ord_{t_v = 0} F(\bm{t})
		\ge
		\min \{ 0, \abs{\overline{v}} + 2 - \deg(v) \}.
		\]
		
		\item \label{item:prop:Gauss_vanish2}
		If $ \abs{\overline{v}} + 2 - \deg(v) > 0 $ for each vertex $ v \in V_{\ge 3} $, then the meromorphic function $ F(f_1; t) $ is holomorphic at $ t=0 $ and it holds
		\begin{equation} \label{eq:F(f_1;0)}
			F(0) 
			= 
			\frac{\bm{e} \left( \abs{V_{\ge 2}}/8 \right)}{\sqrt{2k}^{\abs{V_{\ge 2}}} \prod_{i \in V_1} \sqrt{\abs{w_i}}}
			\sum_{ \mu \in (\Z \smallsetminus k\Z)^{V_{\ge 2}} / 2kS \left(\Z^{V_{\ge 2}} \right) } 
			\bm{e} \left( -\frac{1}{4k} {}^t\!\mu S^{-1} \mu \right) 
			\prod_{v \in V_{\ge 2}} G_v \left( \zeta_{2kM_v}^{\mu_v} \right).
		\end{equation}
	\end{enumerate}
\end{prop}

We obtain the following corollary by \cref{prop:Gauss_vanish} \cref{item:prop:Gauss_vanish1}. 

\begin{cor} \label{cor:Gauss_vanish}	
	Suppose $ \abs{\overline{v}} + 2 - \deg(v) \ge 0 $ for each vertex $ v \in V_{\ge 3} $.
	Let $ e \in \{ \pm 1 \}^{V_{\ge 3}} $.
	Then, $ c_{e, n} = 0 $ holds for $ n \in \Z^{V_{\ge 3}} $ such that $ n_v \le -1 $ for some vertex $ v \in V_{\ge 3} $.
	In particular, for a $ C^\infty $ function $ f \colon \R^{V_{\ge 2}} \to \bbC $ of rapid decay as $ x_v \to \infty $ for each $ v \in V_{\ge 2} $, 
	the limit $ \lim_{t \to +0}F_e(f; t) $ converges and it can be written as
	\[
	\lim_{t \to +0}F_e(f; t) 
	= 
	c_{e, 0}
	=
	e_v^{\deg (v) + \abs{\overline{v}}} F(0),
	\]
	which is independent of $ f $.
\end{cor}

We obtain the following proposition by combining \cref{prop:Gauss_vanish} \cref{item:prop:Gauss_vanish2,prop:WRT_expression}.

\begin{prop} \label{prop:WRT_F(f_1;t)}
	If $ \abs{\overline{v}} + 2 - \deg(v) > 0 $ for any vertices $ v \in V_{\ge 3} $, then it holds that
	\begin{equation}
		\WRT_k(M(\Gamma))
		= 
		\frac{(-1)^{\abs{V_1}} \zeta_{4k}^{-\sum_{v \in V} (w_{v} + 3) - \sum_{i \in V_1} 1/w_i}}
		{2(\zeta_{2k} - \zeta_{2k}^{-1})}
		F(0) .
	\end{equation}
\end{prop}

We obtain our main result by combining \cref{prop:HB_lim,cor:Gauss_vanish,prop:WRT_F(f_1;t)}.

\begin{cor}[\cref{thm:main}] \label{cor:HB_WRT}
	If $ \abs{\overline{v}} + 2 - \deg(v) > 0 $ for any vertices $ v \in V_{\ge 3} $, then it holds that
	\[
	\WRT_k(M(\Gamma))
	=
	\frac{1}{2(\zeta_{2k} - \zeta_{2k}^{-1})}
	\lim_{q \to \zeta_k} \widehat{Z}_\Gamma (q).
	\]
\end{cor}


\subsection{A proof of the main result} \label{subsec:Gauss_vanish_proof}


Finally, we prove \cref{prop:Gauss_vanish}. 
We need the following lemma.

\begin{lem} \label{lem:Gauss_sum_indep}
	Let $ k, M \in \Z \smallsetminus \{ 0 \} $ and $ a, b \in \Z $ be integers such that $ \gcd(a, M) = 1 $.
	For $ \alpha \in \frac{1}{2M}\Z $ such that $ \gcd(2M \alpha, M) = 1 $, the complex number 
	\[
	\sum_{\mu \in \Z/k\Z + \alpha} \bm{e} \left( \frac{M}{k} \left( a \mu^2 + b \mu \right) \right)
	\]
	depends only on $ M\alpha^2, M\alpha \bmod \Z $. 
\end{lem}

This lemma is a little generalisation of \cite[Lemma 4.7 and 4.8]{MM} which is based on the proof of \cite[Theorem 4.1]{BMM_high_depth}.
In this lemma, we remove the assumption $ \gcd(2M, 2M \alpha) = 1 $ in \cite[Lemma 4.7 and 4.8]{MM}.
Although our proof is essentially the same as in \cite[Lemma 4.7 and 4.8]{MM}, we give a proof for the convenience of readers.

\begin{proof}
	\textbf{Case 1}: Suppose $ \gcd(M, k) > 1 $.
	We can write
	\[
	\sum_{\mu \in \Z/k\Z + \alpha} \bm{e} \left( \frac{M}{k} \left( a \mu^2 + b \mu \right) \right)
	=
	\bm{e} \left( \frac{M}{k} \left( a \alpha^2 + b \alpha \right) \right)
	\sum_{\mu \in \Z/k\Z} \bm{e} \left( \frac{1}{k} \left( Ma \mu^2 + (2M \alpha a + Mb) \mu \right) \right).
	\]
	By the assumption, we have $ \gcd(a, M) = 1, \, \gcd(2M \alpha, M) = 1 $ and thus, $ \gcd(2M \alpha a + Mb, M) = 1 $ holds.
	By combining with $ \gcd(M, k) \mid \gcd(Ma, k) $ and $ 1 < \gcd(M, k) \mid M $, we obtain $ \gcd(Ma, k) \nmid 2M \alpha a + Mb $. 
	Thus, it is $ 0 $ by the well-known vanishing result of one-variable Gauss sums (\cite[Lemma 4.3]{MM}).
	
	\textbf{Case 2}: Suppose $ \gcd(M, k) = 1 $.
	Let $ M^* \in \Z $ be an integer such that $ MM^* \equiv 1 \bmod k $.
	For any $ \mu \in \Z $, we have
	\begin{align}
		&M \left( a (\mu + \alpha)^2 + b (\mu + \alpha) \right)
		- M \left( a (\mu + M M^* \alpha)^2 + b (\mu + M M^* \alpha) \right) & & \\
		= \,
		&M \left( (1 - (MM^*)^2) a \alpha^2 + (1 - M M^*) (b + 2a \mu) \alpha \right) & & \\
		\equiv \,
		&(1 - M M^*) \left( (1 + MM^*) a \alpha^2 + b \alpha \right) & &\bmod k\Z.
	\end{align}
	Thus, we obtain
	\begin{align}
		&\sum_{\mu \in \Z/k\Z + \alpha} \bm{e} \left( \frac{M}{k} \left( a \mu^2 + b \mu \right) \right)\\
		= \,
		&\bm{e} \left( \frac{1 - M M^*}{k}  M  \left( (1 + MM^*) a \alpha^2 + b \alpha \right) \right)
		\sum_{m \in \Z/k\Z + M M^* \alpha} \bm{e}\left( \frac{M}{k} \left( a \mu^2 + b \mu \right) \right).
	\end{align}
	Since this sum depends only on $ M\alpha^2 \bmod \Z $ and $ M\alpha \bmod \Z $, we obtain the claim.
\end{proof}

\begin{proof}[Proof of $ \cref{prop:Gauss_vanish} $]
	For \cref{item:prop:Gauss_vanish1}, fix a vertex $ v \in V_{\ge 3} $ and let
	\begin{align}
		F(t) 
		&\coloneqq
		\left( 2k \bm{e} \left( -1/8 \right) \right)^{\abs{V_{\ge 2}}} \sqrt{\det S} F(f_1; t)
		\\
		F_v(t)
		&\coloneqq 
		\sum_{ \mu \in \left. \left( \Z^{V_{\ge 2} \smallsetminus \{ v \} } \oplus k\Z \right) \middle/ 2kS \left(\Z^{V_{\ge 2}} \right) \right. } 
		\bm{e} \left( -\frac{1}{4k} {}^t\!\mu S^{-1} \mu \right) 
		\prod_{v' \in V_{\ge 2}} G_{v'} \left( \zeta_{2kM_{v'}}^{\mu_{v'}} e^{-t_{v'}/2M_{v'}} \right).
	\end{align}
	By \cref{lem:F(f_1;t)}, we can write
	\begin{align}
		&F(t) - F_v(t) \\
		= \,
		&\sum_{ \mu \in \left. \left( \Z^{V_{\ge 2} \smallsetminus \{ v \} } \oplus \left( \Z \smallsetminus k\Z \right) \right) \middle/ 2kS \left(\Z^{V_{\ge 2}} \right) \right. } 
		\bm{e} \left( -\frac{1}{4k} {}^t\!\mu S^{-1} \mu \right) 
		\prod_{v' \in V_{\ge 2}} G_{v'} \left( \zeta_{2kM_{v'}}^{\mu_{v'}} e^{-t_{v'}/2M_{v'}} \right).
	\end{align}
	Since the rational function
	\[
	G_v (\zeta_{2k}^{\mu_v} e^{-t_v/2}) = 
	(\zeta_{2k}^{\mu_v} e^{-t_v/2} - \zeta_{2k}^{-\mu_v} e^{t_v/2})^{2 - \deg(v)} 
	\prod_{i \in \overline{v}} 
	\left(
	\left( \zeta_{2k}^{\mu_v} e^{-t_v/2} \right)^{M_{v} / w_{i}} - \left( \zeta_{2k}^{\mu_v} e^{-t_v/2} \right)^{-M_{v} / w_{i}} 
	\right)
	\]
	is holomorphic at $ t_v = 0 $ for any $ \mu_v \in \Z \smallsetminus k\Z $, the rational function $ F(t) - F_v(t) $ is also holomorphic at $ t_v = 0 $.
	Thus, it suffices to show 
	$ \ord_{t_v = 0} F_v(t)
	\ge
	\abs{\overline{v}} + 2 - \deg(v) $.
	Let 
	\[
	a_v \coloneqq M_v \left( -w_v + \sum_{i \in \overline{v}} \frac{1}{w_i} \right), \quad
	b_v \coloneqq \sum_{v' \in V_{\ge 3}, \, \{ v, v' \} \in E} \mu_{v'}
	\]
	for each $ \mu' \in \Z^{V_{\ge 2} \smallsetminus \{ v \} } $.
	Under this notations, we can write
	\begin{equation} \label{eq:F_v_expression}
		\begin{aligned}
			F_v(t)
			\coloneqq \, 
			&\sum_{ \mu' \in \bigoplus_{v' \in V_{\ge 2} \smallsetminus \{ v \} } \Z / 2k M_{v'} \Z } 
			\bm{e} \left( -\frac{1}{4k} {}^t\!\mu' S^{-1} \mu' \right) 
			\left( 
			\prod_{ v \in V_{\ge 2} \smallsetminus \{ v \} } G_{v'} \left( \zeta_{2kM_{v'}}^{\mu_{v'}} e^{-t_{v'}/2M_{v'}} \right) 
			\right) \\
			&\sum_{ \mu_v \in \Z / 2 M_{v} \Z } 
			\bm{e} \left( -\frac{1}{4k} \left( \frac{k a'_v}{M_v} \mu_v^2 + 2b_v \mu_v \right) \right) 
			G_v \left( \zeta_{2M_v}^{\mu_v} e^{-t_v/2M_v} \right).
		\end{aligned}
	\end{equation}
	Since we have 
	\[
	G_{v} (q)
	=
	(q^{2M_{v}} - 1)^{2 - \deg(v)}
	\sum_{\alpha_{v} \in \calS_{v}} \veps_{v}(\alpha_{v}) q^{2M_{v} \alpha_{v}}
	\]
	by the definition of $ G_v(q) $, the last line of \cref{eq:F_v_expression} is equal to
	\begin{align}
		&(e^{-t_v} - 1)^{2 - \deg(v)}
		\sum_{\alpha_{v} \in \calS_{v}} \veps_{v}(\alpha_{v}) e^{-\alpha_{v} t_v}
		\sum_{ \mu_v \in \Z / 2 M_{v} \Z } 
		\bm{e} \left( -\frac{k a'_v}{4 M_v} \mu_v^2 + \left( \alpha_v - \frac{b_v}{2} \right) \mu_v \right).
	\end{align}
	By applying reciprocity of Gauss sums (\cref{prop:reciprocity}), this can be written as
	\begin{align} \label{eq:alpha_indep}
		&\frac{\bm{e} (1/8) \sqrt{k a_v}}{\sqrt{2 \abs{M_v}}}
		(e^{-t_v} - 1)^{2 - \deg(v)}
		\sum_{\alpha_{v} \in \calS_{v}} \veps_{v}(\alpha_{v}) e^{-\alpha_{v} t_v}
		\sum_{ \mu_v \in \Z / k a_v \Z } 
		\bm{e} \left( -\frac{M_v}{k a'_v} \left( \mu_v + \alpha_v - \frac{b_v}{2} \right)^2 \right).
	\end{align}
	Since the second summation is independent of $ \alpha_v $ by \cref{lem:epsilon} \cref{item:lem:epsilon3}, \cref{lem:bij_calS_v} \cref{item:lem:bij_calS_v3}, and \cref{lem:Gauss_sum_indep}, \cref{eq:alpha_indep} is the multiplication of $ G_v(e^{-t_v / 2M_v}) $ by a constant independent of $ \alpha_v $.
	Thus, we obtain $ \ord_{t_v = 0} F_v(t) \ge \ord_{t_v = 0} G_v(e^{-t_v / 2M_v}) $.
	Since $ \ord_{t_v = 0} G_v(e^{-t_v / 2M_v}) = 2 - \deg(v) + \abs{\overline{v}} $ holds by \cref{lem:G_v_at_1}, the claim holds.
	
	Finally we prove \cref{item:prop:Gauss_vanish2}.
	If $ \abs{\overline{v}} + 2 - \deg(v) > 0 $ for each vertex $ v \in V_{\ge 3} $, then $ G_v(e^{-t_v / 2M_v}) $ has a zero at $ t_v = 0 $ by \cref{lem:G_v_at_1}, thus $ F_v(t) $ also has a zero.
	Thus, we obtain
	\[
	\restrict{F(t)}{t_v = 0}
	=
	\restrict{F(t) - F_v(t)}{t_v = 0}.
	\]
	Fix any vertex $ v' \in V_{\ge 2} \smallsetminus \{ v \} $ and let
	\begin{align}
		F_{v, v'}(t)
		\coloneqq 
		F(t) - F_v(t) -
		\sum_{ \mu \in 
			\left. \left( \Z^{V_{\ge 2} \smallsetminus \{ v, v' \} } \oplus \left( \Z \smallsetminus k\Z \right)^{\{ v, v' \}} \right) 
			\middle/ 2kS \left(\Z^{V_{\ge 2}} \right) \right. } 
		&\bm{e} \left( -\frac{1}{4k} {}^t\!\mu S^{-1} \mu \right) 
		\\
		&\prod_{v' \in V_{\ge 2}} G_{v'} \left( \zeta_{2kM_{v'}}^{\mu_{v'}} e^{-t_{v'}/2M_{v'}} \right).
	\end{align}
	By the same argument in a proof of \cref{item:prop:Gauss_vanish1}, the rational function $ F_{v, v'}(t) $ has a zero at $ t_{v'} = 0 $.
	By induction, we obtain the claim.
\end{proof}


\bibliographystyle{alpha}
\bibliography{quantum_inv_plumbed_work}

\begin{thebibliography}{BKMN21}

\bibitem[AH12]{Andersen-Himpel}
J.~E. Andersen and B.~Himpel.
\newblock The {W}itten--{R}eshetikhin--{T}uraev invariants of finite order
  mapping tori {II}.
\newblock {\em Quantum Topol.}, 3(3-4):377--421, 2012.

\bibitem[AM22]{Andersen-Mistegard}
J.~E. Andersen and W.~Misteg{\aa}rd.
\newblock Resurgence analysis of quantum invariants of {S}eifert fibered
  homology spheres.
\newblock {\em Journal of the London Mathematical Society}, 105(2):709--764,
  2022.

\bibitem[And13]{Andersen}
J.~E. Andersen.
\newblock The {W}itten--{R}eshetikhin--{T}uraev invariants of finite order
  mapping tori {I}.
\newblock {\em J. Reine Angew. Math.}, 681:1--38, 2013.

\bibitem[AS64]{Handbook_formulas}
Milton Abramowitz and Irene~A. Stegun.
\newblock {\em Handbook of mathematical functions with formulas, graphs, and
  mathematical tables}, volume No. 55 of {\em National Bureau of Standards
  Applied Mathematics Series}.
\newblock U. S. Government Printing Office, Washington, DC, 1964.
\newblock For sale by the Superintendent of Documents.

\bibitem[BKM19]{BKM}
K.~Bringmann, J.~Kaszian, and A.~Milas.
\newblock Higher depth quantum modular forms, multiple {E}ichler integrals, and
  {$\mathfrak{sl}_3$} false theta functions.
\newblock {\em Res. Math. Sci.}, 6(2):Paper No. 20, 41, 2019.

\bibitem[BKMN21]{BKMN_False_modular}
K.~Bringmann, J.~Kaszian, A.~Milas, and C.~Nazaroglu.
\newblock Higher depth false modular forms.
\newblock {\em arXiv preprint arXiv:2109.00394}, 2021.

\bibitem[BMM20]{BMM_high_depth}
K.~Bringmann, K.~Mahlburg, and A.~Milas.
\newblock Higher depth quantum modular forms and plumbed 3-manifolds.
\newblock {\em Lett. Math. Phys.}, 110(10):2675--2702, 2020.

\bibitem[BN19]{Bringmann-Nazaroglu}
K.~Bringmann and C.~Nazaroglu.
\newblock A framework for modular properties of false theta functions.
\newblock {\em Res. Math. Sci.}, 6(3):Paper No. 30, 23, 2019.

\bibitem[BW05]{Beasley-Witten}
C.~Beasley and E.~Witten.
\newblock Non-abelian localization for {C}hern-{S}imons theory.
\newblock {\em J. Differential Geom.}, 70(2):183--323, 2005.

\bibitem[Cha16]{Charles}
L.~Charles.
\newblock On the {W}itten asymptotic conjecture for {S}eifert manifolds.
\newblock {\em arXiv:1605.04124}, 2016.

\bibitem[Chu17]{Chun}
S.~Chun.
\newblock A resurgence analysis of the {$SU(2)$} {C}hern-{S}imons partition
  functions on a {B}rieskorn homology sphere {$\Sigma(2,5,7)$}.
\newblock {\em arXiv:1701.03528v1}, 2017.

\bibitem[CM15]{Charles-Marche_II}
L.~Charles and J.~March\'{e}.
\newblock Knot state asymptotics {II}: {W}itten conjecture and irreducible
  representations.
\newblock {\em Publ. Math. Inst. Hautes \'{E}tudes Sci.}, 121:323--361, 2015.

\bibitem[DT07]{DT}
F.~Deloup and V.~Turaev.
\newblock On reciprocity.
\newblock {\em J. Pure Appl. Algebra}, 208(1):153--158, 2007.

\bibitem[FG91]{Freed-Gompf}
D.~S. Freed and R.~E. Gompf.
\newblock Computer calculation of {W}itten's {$3$}-manifold invariant.
\newblock {\em Comm. Math. Phys.}, 141(1):79--117, 1991.

\bibitem[FIMT21]{FIMT}
H.~Fuji, K.~Iwaki, H.~Murakami, and Y.~Terashima.
\newblock Witten–{R}eshetikhin–{T}uraev function for a knot in {S}eifert
  manifolds.
\newblock {\em Communications in Mathematical Physics}, 2021.

\bibitem[GMnP16]{GMP}
S.~Gukov, M.~Mari\~{n}o, and P.~Putrov.
\newblock {R}esurgence in complex {C}hern-{S}imons theory.
\newblock 2016.
\newblock arXiv:1605.07615v2.

\bibitem[GPPV20]{GPPV}
S.~Gukov, D.~Pei, P.~Putrov, and C.~Vafa.
\newblock B{PS} spectra and 3-manifold invariants.
\newblock {\em J. Knot Theory Ramifications}, 29(2):2040003, 85, 2020.

\bibitem[Hik05a]{H_Bries}
K.~Hikami.
\newblock On the quantum invariant for the {B}rieskorn homology spheres.
\newblock {\em Internat. J. Math.}, 16(6):661--685, 2005.

\bibitem[Hik05b]{H_Lattice}
K.~Hikami.
\newblock Quantum invariant, modular form, and lattice points.
\newblock {\em Int. Math. Res. Not.}, (3):121--154, 2005.

\bibitem[Hik06a]{H_Seifert}
K.~Hikami.
\newblock On the quantum invariants for the spherical {S}eifert manifolds.
\newblock {\em Comm. Math. Phys.}, 268(2):285--319, 2006.

\bibitem[Hik06b]{H_Lattice2}
K.~Hikami.
\newblock Quantum invariants, modular forms, and lattice points. {II}.
\newblock {\em J. Math. Phys.}, 47(10):102301, 32, 2006.

\bibitem[HT04]{Hansen-Takata}
S.~K. Hansen and T.~Takata.
\newblock Reshetikhin--{T}uraev invariants of {S}eifert 3-manifolds for
  classical simple {L}ie algebras.
\newblock {\em Journal of Knot Theory and Its Ramifications}, 13(05):617--668,
  2004.

\bibitem[Jef92]{Jeffrey}
L.~C. Jeffrey.
\newblock Chern-{S}imons-{W}itten invariants of lens spaces and torus bundles,
  and the semiclassical approximation.
\newblock {\em Comm. Math. Phys.}, 147(3):563--604, 1992.

\bibitem[LZ99]{LZ}
R.~Lawrence and D.~Zagier.
\newblock Modular forms and quantum invariants of {$3$}-manifolds.
\newblock {\em Asian J. Math.}, 3(1):93--107, 1999.
\newblock Sir Michael Atiyah: a great mathematician of the twentieth century.

\bibitem[MM22]{MM}
A.~Mori and Y.~Murakami.
\newblock Witten--{R}eshetikhin--{T}uraev invariants, homological blocks, and
  quantum modular forms for unimodular plumbing {H}-graphs.
\newblock {\em SIGMA. Symmetry, Integrability and Geometry: Methods and
  Applications}, 18:034, 2022.

\bibitem[MT21]{Matsusaka-Terashima}
T.~Matsusaka and Y.~Terashima.
\newblock Modular transformations of homological blocks for {S}eifert fibered
  homology $3$-spheres.
\newblock {\em arXiv:2112.06210}, 2021.

\bibitem[Neu80]{Neumann_Lecture}
W.~D. Neumann.
\newblock An invariant of plumbed homology spheres.
\newblock In {\em Topology {S}ymposium, {S}iegen 1979 ({P}roc. {S}ympos.,
  {U}niv. {S}iegen, {S}iegen, 1979)}, volume 788 of {\em Lecture Notes in
  Math.}, pages 125--144. Springer, Berlin, 1980.

\bibitem[Neu81]{Neumann_work}
W.~D. Neumann.
\newblock A calculus for plumbing applied to the topology of complex surface
  singularities and degenerating complex curves.
\newblock {\em Trans. Amer. Math. Soc.}, 268(2):299--344, 1981.

\bibitem[Roz94]{Rozansky1}
L.~Rozansky.
\newblock A large {$k$} asymptotics of {W}itten's invariant of {S}eifert
  manifolds.
\newblock In {\em Proceedings of the {C}onference on {Q}uantum {T}opology
  ({M}anhattan, {KS}, 1993)}, pages 307--354. World Sci. Publ., River Edge, NJ,
  1994.

\bibitem[Roz96]{Rozansky2}
L.~Rozansky.
\newblock Residue formulas for the large {$k$} asymptotics of {W}itten's
  invariants of {S}eifert manifolds. {T}he case of {${\rm SU}(2)$}.
\newblock {\em Comm. Math. Phys.}, 178(1):27--60, 1996.

\bibitem[RT91]{Reshetikhin-Turaev}
N.~Reshetikhin and V.~G. Turaev.
\newblock Invariants of {$3$}-manifolds via link polynomials and quantum
  groups.
\newblock {\em Invent. Math.}, 103(3):547--597, 1991.

\bibitem[Wit89]{Witten}
E.~Witten.
\newblock Quantum field theory and the {J}ones polynomial.
\newblock {\em Comm. Math. Phys.}, 121(3):351--399, 1989.

\bibitem[Wu21]{Wu}
D.~H. Wu.
\newblock Resurgent analysis of {$\rm SU(2)$} {C}hern-{S}imons partition
  function on {B}rieskorn spheres {$\Sigma(2, 3, 6n + 5)$}.
\newblock {\em J. High Energy Phys.}, (2):Paper No. 008, 18, 2021.

\bibitem[Zag06]{Zagier_asymptotic}
D.~Zagier.
\newblock The mellin transform and other useful analytic techniques.
\newblock In {\em Appendix to E. Zeidler, Quantum Field Theory {I}: {B}asics in
  Mathematics and Physics. {A} Bridge Between Mathematicians and Physicists},
  pages 305--323. Springer, Berlin, 2006.

\bibitem[Zag10]{Zagier_quantum}
D.~Zagier.
\newblock Quantum modular forms.
\newblock In {\em Quanta of maths}, volume~11 of {\em Clay Math. Proc.}, pages
  659--675. Amer. Math. Soc., Providence, RI, 2010.

\end{thebibliography}

\end{document}